\documentclass[twoside,a4paper,12pt]{article}
\usepackage[top=4.5cm, bottom=4.5cm, inner=3.5cm,outer=3.5cm]{geometry}
\usepackage[utf8]{inputenc}
\usepackage{amsthm}
\usepackage{mathtools}
\usepackage{csquotes}
\usepackage{amsmath}
\usepackage{graphicx}
\usepackage{dsfont}
\usepackage{soul}
\usepackage{amssymb}
\usepackage{multicol}
\usepackage{courier}
\usepackage{listings}
\usepackage{eurosym}
\usepackage{subfiles}
\usepackage{multicol}
\usepackage{pdfpages}
\usepackage{footnote}
\usepackage{color}
\definecolor{myblue}{rgb}{.25, .25, .9}
\definecolor{myred}{rgb}{.6, .3, .3}
\definecolor{myred2}{rgb}{.6, .1, .1}
\definecolor{mygreen}{rgb}{.25, .6, .5}
\usepackage{adjustbox}
\usepackage[colorlinks=true,
            linkcolor=myred2,
            urlcolor=myred,
            citecolor=myred]{hyperref}
\usepackage{pdfpages}
\usepackage[english]{babel}
\usepackage{array}
\usepackage{enumerate}
\usepackage{gensymb}
\usepackage[nottoc]{tocbibind}
\usepackage{lettrine}
\usepackage[normalem]{ulem}
\usepackage[font=small,labelfont=bf]{caption}
\usepackage{cancel}
\usepackage{nomencl}
\usepackage{mathrsfs}
\usepackage{changepage}
\usepackage{amsmath,stmaryrd,graphicx,float}
\numberwithin{equation}{section}

\usepackage{tikz-cd}
\usetikzlibrary{arrows}

\usepackage{subcaption} 

\usepackage{titlesec}

\newtheoremstyle{mystyle}
  {}
  {}
  {\itshape}
  {}
  {\bfseries}
  {.}
  { }
  {}

\theoremstyle{mystyle}

\newtheorem{theorem}{Theorem}[section]

\newtheorem{definition}[theorem]{Definition}

\newtheorem{proposition}[theorem]{Proposition}

\newtheorem{assumption}[theorem]{Assumption} 

\newtheorem{example}[theorem]{Example}
\newtheorem{remark}[theorem]{Remark}

\newcommand\xqed[1]{%
  \leavevmode\unskip\penalty9999 \hbox{}\nobreak\hfill
  \quad\hbox{#1}}
\newcommand\exampleEnd{\xqed{$\circ$}}
\newcommand\remarkEnd{\xqed{$\circ$}}

\usepackage{etoolbox}

\DeclareMathAlphabet{\mathbbold}{U}{bbold}{m}{n}

\titleformat{\section}[runin]
{\bfseries}{\llap{\thesection\hskip 9pt}}{0pt}{}
\titlespacing*{\section}
{0pt}{3.25ex plus 1ex minus .2ex}{1.5ex plus .2ex}

\titleformat{\subsection}[runin]
{\bfseries}{\llap{\thesubsection\hskip 9pt}}{0pt}{}
\titlespacing*{\subsection}
{0pt}{3.25ex plus 1ex minus .2ex}{1.5ex plus .2ex}

\titleformat{\subsubsection}[runin]
{\bfseries}{\llap{\thesubsubsection\hskip 9pt}}{0pt}{}
\titlespacing*{\subsubsection}
{0pt}{3.25ex plus 1ex minus .2ex}{1.5ex plus .2ex}

\titleformat{\paragraph}[runin]
{\bfseries}{\llap{\theparagraph\hskip 9pt}}{0pt}{}
\titlespacing*{\paragraph}
{0pt}{3.25ex plus 1ex minus .2ex}{1.5ex plus .2ex}

\usepackage[titles]{tocloft}

\usepackage{fancyhdr}
\usepackage{lastpage}
\fancypagestyle{plain}{
\fancyhf{}
\fancyhead[LE]{\thepage\quad\textbf{\nouppercase\leftmark}}
\fancyhead[RO]{\textbf{\nouppercase\rightmark}\quad\thepage}
}

\fancypagestyle{basicstyle}{%
\fancyhf{}
\fancyhead[LE]{\thepage}
\fancyhead[RO]{\thepage}
}

\usepackage{algorithm}
\usepackage{algorithmic}[1]

\usepackage{chngcntr}
\counterwithin{figure}{section}

\usepackage[backend=bibtex,
maxbibnames=99,
style=alphabetic,
bibencoding=ascii,
giveninits=true, 
]{biblatex}
\renewbibmacro{in:}{}

\usepackage{rotating}
\usepackage{mathtools}

\newcounter{savecntr}

\title{{\textbf{\Large Asymptotic stability equals exponential stability---\textit{while} you twist your eyes}}}
\author{Wouter Jongeneel\setcounter{savecntr}{\value{footnote}}\thanks{The author is with the KTH Royal Institute of Technology and Digital Futures, Stockholm, Sweden. This note was written while the author was visiting the Division of Decision and Control Systems (DCS) at the KTH Royal Institute of Technology. Their extraordinary hospitality is kindly acknowledged. During that visit, the author was supported by the Risk Analytics and Optimization (RAO) chair at EPFL and by the NCCR \textit{Automation} (grant agreement 51NF40\_180545). The author would like to thank Matthew Kvalheim for helpful discussions and feedback. A special thanks to Elise Cahard, Mikael Johansson and Daniel Kuhn. email: \texttt{wouterjo@kth.se}, website: \url{wjongeneel.nl}.}
}
\date{\small{\today}}

\begin{document}
\maketitle
\thispagestyle{empty}

\begin{abstract}
Suppose that two vector fields on a smooth manifold render some equilibrium point globally asymptotically stable (GAS). We show that there exists a homotopy between the corresponding semiflows such that this point remains GAS along this homotopy.
\end{abstract}

{\footnotesize{
}}


\section{Introduction}
In the context of what we call today \textit{Conley index theory} (see Section~\ref{sec:Conley}), Conley posed the following ``\textit{converse question}'' in the 1970s:
``\textit{To what extent does the homotopy index \emph{[Conley index]} itself determine the equivalence class of
isolated invariant sets which are related by continuation?}''~\cite[p. 83]{ref:conley1978isolated}. Recently, Kvalheim proved that uniquely integrable $C^0$ vector fields, on a $C^{\infty}$ manifold $M$, rendering a compact set $A\subseteq M$ asymptotically stable, are homotopic on an open neighbourhood $U\supseteq A$ such that throughout the homotopy the vector fields do not vanish on $U\setminus A$~\cite[Thm. 1]{ref:kvalheim2022obstructions}. Connecting this result to Conley, a follow-up question (revitalized) by Kvalheim---which is the central question of this note---is the following.\\\\ \\
\noindent\textbf{{Question:}} ``\textit{Are dynamical systems that render a set $A$ asymptotically stable, homotopic through dynamical systems that preserve this notion of stability?}''\\\\
This question, in one form or another, inspired several works, for instance, \cite{ref:reineck1992continuation,ref:mrozek2000conley,ref:JongeneelSchwan2024TAC}.  Here, we will elaborate on commentary by the author in~\cite{ref:JongeneelSchwan2024TAC}.
 Specifically, in the seminal paper ``\textit{Asymptotic stability equals exponential stability, and ISS equals finite energy gain---if you twist your eyes}'' from 1999, Gr\"une, Sontag and Wirth showed that asymptotic stability ``\textit{equals}'' exponential stability in the sense that if an equilibrium point is asymptotically stable under some vector field on $\mathbb{R}^n$, then, there is a suitable change of coordinates rendering this point exponentially stable~\cite{ref:grune1999asymptotic}. Such a change of coordinates is understood to be \textit{instantaneous}. However, by leveraging their work, we show  that asymptotic stability can be \textit{continuously} ``transformed'' into exponential stability, while preserving asymptotic stability throughout the transformation, see Theorem~\ref{thm:GAS:homotopy:semiflow:M}. Differently put, asymptotic stability equals exponential stability---not only \textit{if} you twist your eyes, but \textit{while} you twist your eyes. The key to this is to observe that we can select the transformation from \cite{ref:grune1999asymptotic} to be an \textit{orientation-preserving} homeomorphism on $\mathbb{R}^n$, not just any homeomorphism \textit{cf}.~\cite[Sec. III]{ref:JongeneelSchwan2024TAC}, see the proof of Proposition~\ref{prop:GAS:homotopy:semiflow}.  
 
This result provides a partial solution to Conley's converse question as it turns out that the asymptotically stable systems under consideration can be continuously transformed into the same exponentially stable system and hence, by transitivity, into each other. Concurrently, we discuss extensions to discontinuous vector fields throughout, {plus we illustrate how to go about extensions to ISS (see Section~\ref{sec:ISS})}. We also discuss intimate connections with optimization and optimal transport (\textit{e.g.}, see Example~\ref{ex:invextoconvex} and~\ref{ex:Gaussian:OT}).  

\subsection{Related work}
It can be argued that questions of the form as above, emerged from studies aimed at classifying manifolds, maps and so forth. A successful, yet coarse, resolution has been found in the study of these objects \textit{up to homotopy}, \textit{e.g.}, motivated early on by the fundamental group being homotopy invariant, Hopf's degree theorem, CW complexes, intractability of topological equivalence and more at the intersection of topology and dynamical systems. 

We cannot do justice here to the wealth of work in this area, but inspired by Andronov, Pontryagin, Thom, Peixoto, Birkhoff, Milnor and several others, it was in particular Smale highlighting that this intersection is fruitful for \textit{both} fields, see \cite{ref:hirsch2012topology} for context. Indeed, after seminal work by Smale on the qualitative classification of dynamical systems~\cite{ref:smale1967differentiable}, homotopy results concerned with \textit{Morse-Smale vector fields} appeared, \textit{e.g.}, see~\cite{ref:asimov1975homotopy,ref:newhouse1976,ref:franks1979morse}. Similarly, one can study if gradient vector fields are homotopic through gradient vector fields~\cite{ref:parusinski1990gradient}, see also~\cite{ref:reineck1991continuation,ref:kvalheim2022obstructions} for work close in spirit to ours. 

In control theory, these tools (\textit{i.e.}, certain homotopy invariants) were used to construct necessary conditions for feedback controllers to exist, \textit{i.e.}, if a desirable (closed-loop) dynamical system belongs to a certain homotopy class, then, there must be at least \textit{some} feedback that renders the control system a member of this class, \textit{e.g.}, see~\cite{ref:brockett1983asymptotic,ref:krasnosel1984geometrical,ref:zabczyk1989,ref:Coron1990}.

These results have in common that the original objects are ``\textit{homotoped}'' to something simple and ``\textit{canonical}'' where we do our analysis. 
For instance, in the context of control systems, when our goal is stabilization of the origin on $\mathbb{R}^n$ through feedback, then, the canonical differential equation is $\dot{x}=-x$ and we would like our closed-loop dynamical system to be in some qualitative sense equivalent to this equation. This is precisely how the ``\textit{index condition}'' by Krasnosel'ski\u{\i} and Zabre\u{\i}ko is derived~\cite[Sec. 52]{ref:krasnosel1984geometrical}. Clearly, necessary conditions of this form are only as valuable as the homotopy class is distinctive, \textit{e.g.}, although the index condition is powerful, it cannot differentiate between $\dot{x}=x$ and $\dot{x}=-x$ on $\mathbb{R}^2$. This is precisely what motivates us, we expect that understanding Conley's converse question can lead to stronger necessary conditions for stabilization through continuous feedback.        

Besides possibly stronger necessary conditions for continuous stabilization, the study of Conley's converse question directly relates to understanding topological properties of spaces of stable dynamical systems (as was motivating prior work~\cite{ref:JongeneelSchwan2024TAC}). This is of practical importance as several frameworks in optimal control and reinforcement learning aim to optimize over precisely such a space, \textit{e.g.}, moving from a pre-stabilized system to an optimally controlled system, we point to \cite{ref:Fazel_18,ref:furieri2024learning} and references therein. 

Before we continue, we recall that Conley's converse question is not trivial. 
\begin{example}[Trivial convex combinations can fail]
\begin{upshape}
    Consider a linear differential equation $\dot{x}=A(s) x$ on $\mathbb{R}^2$ parametrized by the matrices 
    \begin{equation*}
        [0,1]\ni s \mapsto A(s):=s \cdot \begin{pmatrix}
            -1 & 10\\ 0 & -1
        \end{pmatrix} + (1-s)\cdot\begin{pmatrix}
            -1 & 0\\ 10 & -1
        \end{pmatrix}.
    \end{equation*}
Both $A(0)$ and $A(1)$ correspond to global asymptotically stable systems, yet, for $s = \tfrac{1}{2}$, the system $\dot{x}=A(s) x$ is unstable. Hence, we cannot simply construct straight-line homotopies between stable vector fields and expect that stability is preserved. Instead, we know from~\cite{ref:JongeneelSchwan2024TAC} that for vector fields with convex Lyapunov functions we should homotope via the canonical ODE $\dot{x}=-x$. Explicitly, consider the following path of vector fields
    defined by 
\begin{equation*}
    H(x;s) := 
\begin{pmatrix}
    -1 & \max\{0,(1-2s)10\}\\
    \max\{0,(2s-1) 10\} & -1
\end{pmatrix}x, \quad  s\in [0,1].
\exampleEnd
\end{equation*}
\end{upshape}
\end{example}

\paragraph*{Notation} For $r\in \mathbb{N}\cup \{\infty\}$, $C^r(U;V)$ denotes the set of $C^r$-smooth functions from $U$ to $V$. The inner product on $\mathbb{R}^n$ is denoted by $\langle \cdot , \cdot \rangle$ and $\mathbb{S}^{n-1}:=\{x\in \mathbb{R}^n:\|x\|_2=1\}$. 
The identity map $p\mapsto p$ on a space $M$ is denoted by $\mathrm{id}_M$. 
Given a function $V:M\to \mathbb{R}$, then, $V^{-1}(c):=\{p\in M: V(p)=c\}$. We typically use $\varphi$ to denote a (semi)flow. If this flow comes from a vector field $X$, we write $\varphi(\cdot;X)$, similarly, if $\varphi$ is parametrized by $s\in [0,1]$ we write $\varphi(\cdot;s)$, that is, we overload the meaning of $\varphi$. A homotopy $[0,1]\times M\ni(s,p)\mapsto H(s,p)$ is said to be an \textit{{isotopy}} when $p\mapsto H(s,p)$ is a topological embedding for all $s\in[0,1]$. A homeomorphism $\psi:\mathbb{R}^n\to \mathbb{R}^n$ is said to be a \textit{{stable homeomorphism}} when it is a finite composition of homeomorphisms that equal the identity map on some---not necessarily the same---non-empty open subset of $\mathbb{R}^n$. We denote by $\mathrm{Homeo}(\mathbb{R}^n;\mathbb{R}^n)$ the group of homeomorphisms $\psi:\mathbb{R}^n\to \mathbb{R}^n$. Additionally, we denote by $\mathrm{Homeo}^+(\mathbb{R}^n;\mathbb{R}^n)\subset \mathrm{Homeo}(\mathbb{R}^n;\mathbb{R}^n)$ the subgroup of all \textit{orientation-preserving}\footnote{Regarding  \textit{orientation}, we point to~\cite{ref:guillemin2010differential,Lee2} for the smooth case. The topological definition relies on algebraic topology, see~\cite[Sec. 3.3]{Hatcher}.} homeomorphisms. A continuous function $\alpha:\mathbb{R}_{\geq 0}\to \mathbb{R}_{\geq 0}$ is of class $\mathcal{K}$ when $\alpha$ is strictly increasing and $\alpha(0)=0$. If, additionally $\lim_{s\to +\infty}\alpha(s)=+\infty$, then, $\alpha$ is of class $\mathcal{K}_{\infty}$. At last, a continuous function $\beta:\mathbb{R}_{\geq 0}\times \mathbb{R}_{\geq 0} \to \mathbb{R}_{\geq 0}$ is said to be of class $\mathcal{KL}$ when $s\mapsto \beta (s,t)$ is of class $\mathcal{K}$ and $t\mapsto \beta(s,t)$ is decreasing and such that $\lim_{t\to +\infty}\beta(s,t)=0$.   


\subsection{Dynamical systems}
We study \textit{semi-dynamical systems} comprised of the triple $\Sigma:=({M}^n,\mathbb{R}_{\geq 0},\varphi)$. Here,  ${M}^n$ will be a smooth $n$-dimensional manifold (frequently, diffeomorphic to $\mathbb{R}^n$, which we denote by ${M}^n\simeq_d \mathbb{R}^n$,) and $\varphi:\mathbb{R}_{\geq 0}\times {M}^n\to {M}^n$ is a (global) \textit{{semiflow}}, that is, a continuous map that satisfies for any $p\in {M}^n$: $(i)$ $\varphi(0,p)=p$; and $(ii)$ $(\varphi(s,\varphi(t,p))=\varphi(t+s,p)$ $\forall s,t \in \mathbb{R}_{\geq 0}:=\{t\in \mathbb{R}:t\geq 0\}$. We will usually write $\varphi^t(\cdot)$ instead of $\varphi(t,\cdot)$. In particular, we study (forward complete) semiflows generated by \textit{{vector fields}} over ${M}^n$, that is, when $\varphi$ satisfies
\begin{align}
\label{equ:dyn:sys:M}
\frac{\mathrm{d}}{\mathrm{d}\tau}\varphi^{\tau}(p)|_{\tau=t} = X(\varphi^t(p)),\quad \forall (t,p)\in \mathbb{R}_{\geq 0}\times M^n,  
\end{align}
where $X$ is a continuous section, that is, the map $X:M^n\to TM^n$ is continuous and satisfies $\pi\circ X = \mathrm{id}_{M^n}$ for $\pi$ the projection $(p,v)\mapsto \pi(p,v)=p$. When $M^n=\mathbb{R}^n$, the identification $T\mathbb{R}^n\simeq_d \mathbb{R}^n\times \mathbb{R}^n$ allows for discussing vector fields as self-maps of $\mathbb{R}^n$. On $\mathbb{R}^n$, we will frequently work with the ``\textit{canonical}'' ODE $\dot{x}=-x$, but also with the identity map $x\mapsto x$. To avoid confusion, we denote the canonical vector field by $-\partial_x$, and not $-\mathrm{id}_{\mathbb{R}^n}$.

The focus on \textit{semi}flows instead of flows allows us to look at sufficiently regular \textit{discontinuous} vector fields. This is relevant, as the introduction of feedback usually results in a closed-loop vector field that cannot be assumed to be continuous (\textit{e.g.}, think of optimal control\footnote{One might also think of topological obstructions, however, the type of discontinuities we consider do not allow for overcoming those obstructions, in general~\cite{ref:ryan1994brockett}.}). This choice of setting is also motivated by other recent work. For instance, the research question in~\cite{ref:ozaslan2024exponential} is: ``\textit{Given an exponentially stable optimization algorithm, can it be modified to obtain a finite/fixed-time stable algorithm?}''. {Some sufficient conditions are provided in \cite{ref:ozaslan2024exponential}}, we will show a generic, yet less constructive, viewpoint.  

Let $X$ be some vector field on $\mathbb{R}^n$, possibly discontinuous. To study $X$, we usually pass to some \textit{differential inclusion}
\begin{equation}
\label{equ:DI}
\dot{x}\in F(x), 
\end{equation}
where the set-valued map $F:\mathbb{R}^n\to 2^{\mathbb{R}^n}$, with $2^S$ denoting the \textit{power set} of a set $S$, is in some precise sense related to $X$. The rationale is to pass from an irregular single-valued map, to a more regular set-valued map. 

Let $\lambda^d$ denote the Lebesgue measure on $\mathbb{R}^d$, then, solutions $\mathcal{I}\ni t\mapsto \xi(t)$ to these differential inclusions, are \textit{absolutely continuous} (AC) on each compact subinterval of $\mathcal{I}$ and such that $\dot{\xi}(t)\in F(\xi(t))$ for $\lambda^1$-\textit{a.e.} $t\in \mathcal{I}$. Typically, $F$ is assumed to be upper semi-continuous and compact, convex valued. With those assumptions in mind, then, under mild conditions on $X$, a valuable solution framework follows by applying \textit{Filippov's operator} $\mathcal{F}$, that is,
\begin{equation*}
    x\mapsto \mathcal{F}[X](x)  := \bigcap_{\delta>0}\bigcap_{N\subset \mathbb{R}^n:\lambda^n(N)=0}\overline{\mathrm{co}}\,X\left( B^n(x,\delta)\setminus N\right),
\end{equation*}
for $\overline{\mathrm{co}}(\cdot)$ the closure of the convex hull and $B^n(x,\delta):=\{z\in \mathbb{R}^n:\|x-z\|_2<\delta\}$. Then, solutions to $\dot{x}\in \mathcal{F}[X](x)$ are understood to be ``\textit{generalized}'' solutions to $\dot{x}=X(x)$, called \textit{Filippov solutions}, and such that $\dot{\xi}(t)\in \mathcal{F}[X](\xi(t))$ for $\lambda^1$-a.e. $t\in \mathrm{dom}(\xi)$. For references, see \cite{ref:filippov1988differential}, \cite[Ch. 1]{ref:bacciotti2005liapunov} and~\cite{ref:cortes2008discontinuous}.

\subsection{Stability}
In this section we will characterize stability under~\eqref{equ:DI}. Starting with the regular case, for simplicity, let the vector field $F$ be single-valued and smooth. In that case, $F$ generates a flow, denoted $\varphi(\cdot;F)$. A point $x^{\star}\in \mathbb{R}^n$ is an \textit{equilibrium point} of $F$ when $F(x^{\star})=0$ or equivalently $\varphi^t(x^{\star};F)=x^{\star}$ $\forall t\in \mathbb{R}$. We will set $x^{\star}:=0$, unless stated otherwise. Then, $0$ is said to be \textit{{globally asymptotically stable}} (GAS) (under $F$) when 
\begin{enumerate}[(s.i)]
    \item \label{prop:i:s}$0$ is \textit{Lyapunov stable}, that is, for any open neighbourhood $U_{\varepsilon}\ni 0$ there is an open set $U_{\delta}\subseteq U_{\varepsilon}$ such that $\varphi^t(U_{\delta};F)\subseteq U_{\varepsilon}$ $\forall t\in \mathbb{R}_{\geq 0}$;
    \item $0$ is \textit{globally attractive}, that is, $\lim_{t\to+\infty}\varphi^t(x_0;F)=0$ for all $x_0\in \mathbb{R}^n$.
\end{enumerate}
{
\begin{remark}[On global stability]
\begin{upshape}
Although we merely work with equilibria and $\mathbb{R}^n$ (or $M^n\simeq_d \mathbb{R}^n$) we take a topological approach akin to~\cite{ref:bhatiahajek2006local} and hence GAS is defined as above. One can turn Property~(s.\ref{prop:i:s}) truly global, however, \textit{e.g.}, see~\cite{ref:wilson1969smoothing,ref:lin1996smoothV2}. Specifically, let $B_r(x)$ be some $r$-metric ball centered at $x$, then, require the existence of $\delta\in \mathcal{K}_{\infty}$ such that for any $\varepsilon\geq 0$ we have $\varphi^t(B_{\delta(\varepsilon)}(x))\subseteq B_{\varepsilon}(x)$ $\forall t\in \mathbb{R}_{\geq 0}$. In general these definitions are not equivalent, however, for $0$ being GAS on $\mathbb{R}^n$, they are~\cite{ref:andriano1997global}.
\remarkEnd
\end{upshape}
\end{remark}
}
Due to the work of Lyapunov~\cite{ref:liapunov1892general}, we know that to reason about stability, it is worthwhile to look for ``\textit{potential functions}'' that capture stability, illustrated by the fact that his theory effectively replaced the definitions of stability. Specifically, we look for $V\in C^{\infty}(\mathbb{R}^n;\mathbb{R}_{\geq 0})$ satisfying:
\begin{enumerate}[(V.i)]
    \item \label{prop:i:V}$V(x)>0$ for all $x\in \mathbb{R}^n\setminus\{0\}$ and $V(0)=0$;
        \item \label{prop:iii:V}$V$ is \textit{radially unbounded}, that is, $V(x)\to +\infty$ for $\|x\|\to+\infty$; and
    \item \label{prop:ii:V}$\langle \nabla V(x),F(x) \rangle <0$ for all $x\in \mathbb{R}^n\setminus\{0\}$.
\end{enumerate}
Property~(V.\ref{prop:iii:V}) implies sublevel set compactness. We call such a function a (smooth, strict and proper) Lyapunov function (with respect to the pair $(F,0)$). This work is about GAS, so we will omit ``\textit{strict}'' and ``\textit{proper}'' from now on. Then, based on converse theory, \textit{e.g.}, see~\cite{ref:Massera1956,ref:kurzweil1963inversion,ref:wilson1969smoothing,ref:fathi2019smoothing}, we can appeal to the celebrated theorem stating that $0$ is GAS if and only if there is a (corresponding) smooth Lyapunov function~\cite[Thm.~2.4]{ref:bacciotti2005liapunov}. 

A generalization of the above to differential inclusions~\eqref{equ:DI} is as follows. 
\begin{definition}[Strong Lyapunov pairs~{\cite[Def. 1.1]{ref:clarke1998asymptotic}}]
\label{def:stong:Lyapunov:pair}
A pair of continuous functions $(V,W)$, with $V\in C^{\infty}(\mathbb{R}^n;\mathbb{R}_{\geq 0})$ and $W\in C^{\infty}(\mathbb{R}^n\setminus \{0\};\mathbb{R}_{\geq 0})$, is said to be a $C^{\infty}$ strong Lyapunov
pair for $F$, as in~\eqref{equ:DI}, provided that $F$ is upper semi-continuous and compact, convex valued, plus, the following hold:
\begin{enumerate}[(i)]
    \item $V(x)>0$ and $W(x)>0$ on $\mathbb{R}^n\setminus\{0\}$, with $V(0)=0$;
    \item The sublevel sets of $V$ are compact; and
    \item $\max_{v\in F(x)}\langle \nabla V(x),v\rangle \leq -W(x)$ on $\mathbb{R}^n\setminus\{0\}$.
\end{enumerate}
\end{definition}

The preposition ``\textit{strong}'' comes from the fact that we look at all solutions satisfying~\eqref{equ:DI}. 
For further references, and generalizations of Lyapunov's stability theory, we point the reader to~\cite{ref:bhatia1970stability,ref:sontag2013mathematical,ref:bacciotti2005liapunov,ref:bhatiahajek2006local,ref:goebel2012hybrid}. 

Before closing this section, we explicitly illustrate why working with $C^{\infty}$ or even $C^0$ vector fields is arguably overly restrictive when interested in qualitative stability questions. We provide two examples that will return. 
\begin{example}[A semiflow corresponding to a vector field with bounded discontinuities]
\label{ex:X1:semiflow}
\begin{upshape}
    Consider the following \textit{discontinuous} vector field on $\mathbb{R}$:
    \begin{equation}
    \label{equ:discont:F}
        \dot{x} = X_1(x) := -\mathrm{sgn}(x),
    \end{equation}
    with $\mathrm{sgn}(0):=0$. Now, consider the map $\varphi_1:\mathbb{R}_{\geq 0}\times \mathbb{R}\to \mathbb{R}$ defined through
        \begin{equation}
        \label{equ:semiflow:sign}
        (t,x) \mapsto \varphi_1^{t}(x):= \begin{cases}
            \min\{0,x+t\} \quad & x<0\\
            0 \quad & x=0\\
            \max\{0,x-t\} \quad & x>0
        \end{cases}.
    \end{equation}
    One can check that $\varphi_1$ is a semiflow, describing a solution (\textit{e.g.}, in the sense of Filippov) to~\eqref{equ:discont:F}. 
Regarding stability, consider the $C^{\infty}$ Lyapunov function $x\mapsto V_1(x):=\tfrac{1}{2}x^2$ and see that $\nabla V_1(x)X_1(x) = -|x|<0$ on $\mathbb{R}\setminus\{0\}$. This already shows that the existence of a smooth Lyapunov function, asserting that the origin is GAS, does not imply the existence of a \textit{flow}, nor does it imply that convergence to $0$ is merely \textit{asymptotic}. Now let $X_2(x):=-\partial_x$, define $[0,1]\ni s \mapsto X(\cdot;s):=(1-s)X_1+sX_2$ and see that for and $s\in [0,1]$ we have $\nabla V(x)X(x;s) <0$ on $\mathbb{R}\setminus \{0\}$. 
The semiflow corresponding to $X(\cdot;s)$ becomes
    \begin{equation*}
        (t,x) \mapsto \varphi^{t}(x;s):= \begin{cases}
            \min\{0,e^{-s t}x+(1-e^{-s t})(1-s)/s\} \quad & x<0\\
            0 \quad & x=0\\
            \max\{0,e^{-s t}x+(e^{-s t}-1)(1-s)/s\} \quad & x>0
        \end{cases}.
    \end{equation*}
Then, $s\mapsto \varphi(\cdot;s)$ parametrizes a homotopy, along semiflows such that $0$ remains GAS. For $\lim_{s\to 0^+}\varphi(\cdot;s)$, use $e^{-s t}=1-s t+\sum^{\infty}_{n=2}((-s t)^n/n!)$. 
    \exampleEnd
    \end{upshape}
\end{example}

\begin{example}[An irregular gradient flow]
\label{ex:irr:nabla:V}
\begin{upshape}
Let $\gamma\in \mathcal{K}_{\infty}$ be smooth on $(0,+\infty)$ and such that $\gamma(s)/\gamma'(s)\geq s$ (\textit{e.g.}, $s\mapsto \gamma (s)=s^{1/2}$). Now consider the function $x\mapsto V_{\gamma}(x):=\gamma(\|x\|_2)$ and construct the vector field 
    \begin{equation}
    \label{equ:discont:F4}
        \dot{x} = X_3(x) := \begin{cases}
            -\nabla V_{\gamma}(x) \quad & x\in \mathbb{R}^n\setminus \{0\}\\
            0 \quad & \text{otherwise}
        \end{cases}.
    \end{equation}
It seems we cannot appeal to Filippov's framework as {$X_3$ is not locally essentially bounded}. 
However, as $\nabla V_{\gamma}(x) = \gamma'(\|x\|_2)(x/\|x\|_2)$ we can study~\eqref{equ:discont:F4} directly. Decompose $x\in \mathbb{R}^n\setminus\{0\}$ as $x=\|x\|_2\cdot (x/\|x\|_2)=:r\cdot u$, then, $\dot{r}=-\gamma'(r)$ while $\dot{u}=0$. 
In general $\gamma'(s)>0$ $\forall s> 0$, need not be true, but suppose this is true for our choice of $\gamma$, \textit{e.g.}, pick $s\mapsto\gamma(s)=s^{1/2}$.   
Now, suppose that $s\mapsto 1/\gamma'(s)$ is of class $\mathcal{K}$ on $\mathbb{R}_{\geq 0}$, we can define $r\mapsto\Gamma(r):=\int^r_0 1/\gamma'(\varrho) \mathrm{d}\varrho$, which is now of class $\mathcal{K}_{\infty}$ and hence invertible on $\mathbb{R}_{\geq 0}$. Under the aforementioned assumptions, the semiflow corresponding to $X_3$ becomes
    \begin{equation*}
        (t,x) \mapsto \varphi_3^{t}(x):= \begin{cases}
            \min\{0,\Gamma^{-1}(\Gamma(\|x\|_2)-t)\}(x/\|x\|_2) \quad & \Gamma^{-1}(t)\leq \|x\|_2\\
            0 \quad & \mathrm{otherwise}
        \end{cases}.
        \exampleEnd 
    \end{equation*} 
\end{upshape}
\end{example}

\section{Further comments on related work}
\label{sec:further}
Recently, we showed that when $0\in \mathbb{R}^n$ is GAS under a continuous vector field $X$, and if this can be asserted using a $C^1$ \textit{convex} Lyapunov function, then, $X$ is straight-line homotopic to $-\partial_x$, such that the origin remains GAS along the homotopy~\cite{ref:JongeneelSchwan2024TAC}. This is a convenient result, but not a general one. 

Earlier, the following homotopy of vector fields, appeared in several works (\textit{e.g.}, see \cite[p. 1603]{ref:ryan1994brockett}, \cite[Thm. 21]{ref:sontag2013mathematical}, \cite[p. 291]{ref:coron2007control} and \cite[Ex. 3.4]{ref:jongeneel2023topological}):  
\begin{equation}
\label{equ:flow:homotopy:sontag:coron}
    (s,x) \mapsto H(s,x) := \begin{cases}
        X(x) \quad &\text{if }s=0\\
        -x \quad &\text{if }s=1\\
        \frac{1}{s}(\varphi^{s/(1-s)}(x;X)-x)\quad &\text{if }s\in(0,1)
    \end{cases}.
\end{equation}

Unfortunately, for~\eqref{equ:flow:homotopy:sontag:coron}, $0\in \mathbb{R}^n$ is not known to be GAS along the path $[0,1]\ni s\mapsto H(s,x)$.
The scalar and linear cases can be understood, however. 

\begin{example}[Stability-preserving homotopies for $n=1$]
\begin{upshape}
As for any $t\in \mathbb{R}_{> 0}$, the only fixed point of $\varphi^{t}(\cdot;X)$ is $0$, it follows that for $n=1$, the homotopy~\eqref{equ:flow:homotopy:sontag:coron} preserves stability since the sign of $x\mapsto H(s,x)$ cannot flip, for otherwise, $\varphi^{s/(1-s)}(x;X)=x$ must hold for some $s\in (0,1)$ and $x\neq 0$.  
\exampleEnd
\end{upshape}
\end{example}

\begin{example}[Stability preserving homotopies for linear ODEs]
\begin{upshape}
    Let $0\in \mathbb{R}^n$ be GAS under $\dot{x}=X(x):=Ax$, for $A=:TJT^{-1}$ the Jordan form decomposition of $A$. Then it follows that $\frac{1}{s}(\varphi^{s/(1-s)}(x;X)-x) = \frac{1}{s}T( e^{s/(1-s) J}-I_n) T^{-1}x$ such that stability is preserved throughout the homotopy~\eqref{equ:flow:homotopy:sontag:coron}. 
    \exampleEnd
    \end{upshape}
\end{example}

Either way, the homotopy~\eqref{equ:flow:homotopy:sontag:coron} does already show that there is \textit{a} homotopy that does not introduce new equilibrium points, along the homotopy. Indeed, this has been generalized recently by Kvalheim to compact attractors on manifolds~\cite{ref:kvalheim2022obstructions}. On the other hand, it is known that there is no reason why stability \textit{must} be preserved along such a homotopy. In the spirit of~\cite[Sec. 4.1]{ref:eliashberg2002introduction}, consider the following family of linear vector fields on $\mathbb{R}^2$: 
\begin{equation*}
    \dot{x} = X(x;s):=R(s) x,\quad [0,1]\ni s \mapsto R(s):=\begin{pmatrix}
        \cos(s\pi) & -\sin(s\pi)\\
        \sin(s\pi) & \cos(s\pi)
    \end{pmatrix},
\end{equation*}
parametrizing a homotopy from $X(\cdot;0)=\partial_x$ to $X(\cdot;1)=-\partial_x$ along non-vanishing vector fields on $\mathbb{R}^n\setminus \{0\}$. Note, the (Hopf) indices~\cite[p. 32]{ref:milnor65} of these ``\textit{qualitatively opposite}'' vector fields are equal, \textit{i.e.}, $\mathrm{ind}_0(\partial_x)=1^n=(-1)^n=\mathrm{ind}_0(-\partial_x)$. 
It is this weakness of existing homotopy-invariants that we hope to overcome by studying more restrictive equivalence classes. 

\section{Stability preserving homotopies}
\label{sec:main}
{Now, we will address our central research question in case $A$ is an equilibrium point. First, we consider $0\in \mathbb{R}^n$ being GAS under some appropriately regular vector field on $\mathbb{R}^n$ and eventually generalize to smooth manifolds.}
{
\begin{assumption}[Vector field regularity on $\mathbb{R}^n$]
\label{ass:reg:X}
Our vector fields are locally essentially bounded, possibly set-valued at $0$ and locally Lipschitz on $\mathbb{R}^n\setminus\{0\}$. 
\end{assumption}}
If $X$ satisfies Assumption~\ref{ass:reg:X}, $\mathcal{F}[X]$ is upper semi-continuous and compact, convex valued, which allows for a \textit{smooth} converse Lyapunov theory, \textit{e.g.}, see \cite{ref:clarke1998asymptotic,ref:goebel2012hybrid}. {We focus on Filippov's framework, but in what follows, one may replace  $\mathcal{F}[X]$ with any vector field $F$ that complies with Definition~\ref{def:stong:Lyapunov:pair}.} 
  
\begin{proposition}[Strong global asymptotic stability and homotopic semiflows on $\mathbb{R}^n$]
\label{prop:GAS:homotopy:semiflow}
    Let $0\in \mathbb{R}^n$, for $n \neq 5$, be strongly GAS, in the sense of Filippov, under $\dot{x}\in \mathcal{F}[X](x)$, with $X$ satisfying Assumption~\ref{ass:reg:X}. Then, 
    \begin{enumerate}[(i)]
        \item for any $t\geq 0$ we have that the time-$t$ map of the semiflow generated by $\mathcal{F}[X]$, that is, $\varphi^{t}(\cdot;\mathcal{F}[X])$, is homotopic to $\varphi^{t}(\cdot;-\partial_x)$, along time-$t$ maps corresponding to semiflows such that $0$ is strongly GAS; and
        \item in particular, the semiflow $\varphi(\cdot;\mathcal{F}[X])$ is homotopic to $\varphi(\cdot;-\partial_x)$, along semiflows that preserve $0$ being strongly GAS.
    \end{enumerate}
\end{proposition}
\begin{proof}
    The proof proceeds in 5 steps. First we show that the time-$t$ map $\varphi^{t}(\cdot;\mathcal{F}[X])$ can be homotoped along semiflows to a gradient flow $\varphi^{t}(\cdot;-\nabla V)$, such that along the homotopy $0$ remains strongly GAS. By exploiting symmetry we show in Step (ii) that an extension of~\cite{ref:grune1999asymptotic} allows for showing that for any $\gamma\in \mathcal{K}_{\infty}$ there is a $T\in \mathrm{Homeo}^+(\mathbb{R}^n;\mathbb{R}^n)$ (not just $T\in \mathrm{Homeo}(\mathbb{R}^n;\mathbb{R}^n)$)  such that $V(T^{-1}(x))=\gamma(\|x\|_2)$. Then, it follows that $T$ is homotopic to $\mathrm{id}_{\mathbb{R}^n}$ along a path in $\mathrm{Homeo}^+(\mathbb{R}^n;\mathbb{R}^n)$. In Step (iii), we use this map $T$ to homotope $\varphi^{t}(\cdot;-\nabla V)$ to a semiflow, denoted $\widetilde{\varphi}$, that has $V\circ T^{-1}$ as its Lyapunov function. Again, $0$ remains strongly GAS along the path. Next, we show in Step (iv) that $\widetilde{\varphi}^{t}$ can be homotoped to $\varphi^{t}(\cdot;-\partial_x)$. We conclude in Step (v).

    (i) Since $0$ is strongly GAS under $\dot{x}\in \mathcal{F}[X](x)$ and $X$ satisfies Assumption~\ref{ass:reg:X}, there is a $C^{\infty}$ Lyapunov pair $(V,W)$, that certifies stability of $0\in \mathbb{R}^n$~\cite[Thm. 1.3]{ref:clarke1998asymptotic}, under any Filippov solution. Then, construct the homotopy $H:[0,1]\times \mathbb{R}^n\setminus\{0\}\to \mathbb{R}^n\setminus\{0\}$ through non-vanishing vector fields, defined by $(s,x)\mapsto H(s,x):=(1-s)\mathcal{F}[X](x)-s\nabla V(x)$. As for any $s\in [0,1]$, $(1-s)\mathcal{F}[X]-s\nabla V$ satisfies Assumption~\ref{ass:reg:X}, plus $\langle \nabla V(x), H(s,x)\rangle \leq -(1-s)W(x)-s\| \nabla V(x)\|_2^2$ $\forall\, x\in \mathbb{R}^n\setminus\{0\}$,
    we find that $0$ is strongly GAS, in the sense of Filippov, under $\dot{x}\in (1-s)\mathcal{F}[X](x)-s\nabla V(x)$, for any fixed $s\in [0,1]$. To construct a homotopy on the level of semiflows, we follow Hartman~\cite[p. 93]{hartman2002ordinary} and consider the following extended system: 
    \begin{equation*}
        \Sigma_H :\left\{
\begin{aligned}
   \dot{s}&=0,\\
    \dot{x}&\in (1-s)\mathcal{F}[X](x)-s\nabla V(x). 
\end{aligned}
        \right.
    \end{equation*}     
As $V$ has compact sublevel sets (see Definition~\ref{def:stong:Lyapunov:pair}), any Filippov solution to $\Sigma_H$ is defined for all $t\geq 0$, \textit{e.g.}, see~\cite[Thm. 3.3]{ref:khalil2002nonlinear}.
    Now since $(s,x)\mapsto H(s,x)$ is locally Lipschitz away from $0$, we can appeal to the \textit{Picard-Lindel\"of theorem}~\cite[Thm. 2.2]{ref:teschl2012ordinary} and hence, by strong asymptotic stability~\cite[Def. 2.1]{ref:clarke1998asymptotic} \textit{cf}.~Example~\ref{ex:X1:semiflow}, any solution to $\Sigma_H$ is also unique, which allows us to define the time-$t$ map $(s,x)\mapsto\varphi^{t}((s,x);\Sigma_H)$ for any $t\geq 0$. In its turn, by considering the path $[0,1]\ni s\mapsto \varphi^{t}((s,\cdot);\Sigma_H)$, we see that this time-$t$ map $\varphi^{t}$ defines a homotopy from the semiflow under $\mathcal{F}[X]$ to the (semi)flow under $-\nabla V$. Specifically, we have that $\varphi^{t}((0,x);\Sigma_H)=(0,\varphi^{t}(x;\mathcal{F}[X]))$, $\varphi^{t}((1,x);\Sigma_H)=(1,\varphi^{t}(x;-\nabla V))$ and $\varphi^{t}((s,0);\Sigma_H)=(s,0)$ $\forall s\in [0,1]$. Now, for any fixed $s\in [0,1]$, overload notation and define the time-$t$ map $\varphi^{t}(\cdot;(s,\Sigma_H)):=\pi_{2:n+1}\circ \varphi^{t}((s,\cdot);\Sigma_H)$, for $\pi_{2:n+1}$ the projection on the last $n$ coordinates. It follows that this map checks out as a semiflow, since $\pi_{2:n+1}$ is continuous and for any $x\in \mathbb{R}^n$ we have that
    \begin{equation*}
        \begin{aligned}
            \text{$(i)$ }\,&\varphi^0(x;(s,\Sigma_H))=x;\text{and}\\
            \text{$(ii)$}\,&\varphi^{t_2}(\varphi^{t_1}(x;(s,\Sigma_H));(s,\Sigma)) = \varphi^{t_2+t_1}(x;(s,\Sigma_H)),\quad \forall t_1,t_2\in \mathbb{R}_{\geq 0}.
        \end{aligned}
    \end{equation*}

(ii) Gr\"une, Wirth and Sontag showed that for our Lyapunov function $V$, we can find a $T\in \mathrm{Homeo}(\mathbb{R}^n;\mathbb{R}^n)$, with $T(0)=0$, such that $V(T^{-1}(x))=\gamma(\|x\|_2)$, for some $\gamma\in \mathcal{K}_{\infty}$, smooth on $(0,+\infty)$~\cite[Prop. 1]{ref:grune1999asymptotic}. We recall their construction $T$. Define $(t,x)\mapsto\psi(t,x)$ to be the flow with respect to
\begin{equation*}
    \dot{x} = \frac{\nabla V(x)}{\|\nabla V(x)\|_2^2}.
\end{equation*}
It follows that, on the domain of $\psi$, $V(\psi(t,x))=V(x)+t$. Now fix some $c>0$, then define the map $\pi_c:\mathbb{R}^n\setminus \{0\}\to V^{-1}(c)$ by $x\mapsto \pi_c(x):=\psi(c-V(x),x)$, that is, starting from $x\in \mathbb{R}^n\setminus \{0\}$ you flow---either backward or forward---along $\psi$ until you hit $V^{-1}(c)$. Then, due to initial work by Wilson~\cite{ref:wilson1967structure}, we know that $V^{-1}(c)\simeq_h \mathbb{S}^{n-1}$, however, now---that is, after~\cite{ref:grune1999asymptotic} was published, the resolution of the Poincar\'e conjecture\footnote{Specifically, Perelman provided the final step $(\mathbb{S}^3)$ in proving the \textit{generalized Poincar\'e conjecture} in $\mathsf{Top}$. For some historical comments, see~\cite{ref:stillwell2012poincare}.} implies there must be a homeomorphism $S:V^{-1}(c)\to \mathbb{S}^{n-1}$, for any $n\geq 1$. Next, define the map $Q:=S\circ \pi_c:\mathbb{R}^n\setminus\{0\}\to \mathbb{S}^{n-1}$ and eventually the map $T:\mathbb{R}^n\to \mathbb{R}^n$ by
\begin{equation*}
    x\mapsto T(x) :=\begin{cases}
        \gamma^{-1}(V(x))Q(x) \quad & x\in \mathbb{R}^n\setminus\{0\}\\
        0 \quad &\text{otherwise},
    \end{cases}
\end{equation*}
Now, it turns out that the particular choice of the homeomorphism $S:V^{-1}(c)\to \mathbb{S}^{n-1}$ is irrelevant for the construction of $T$ as in~\cite[Prop. 1]{ref:grune1999asymptotic}, as such, we select $S$ to enforce $T\in \mathrm{Homeo}^+(\mathbb{R}^n;\mathbb{R}^n)$. If $T\in \mathrm{Homeo}^+(\mathbb{R}^n;\mathbb{R}^n)$, we are done, if not, we can compose $S$ with a map reflecting a single coordinate, denoted by $\rho:\mathbb{S}^{n-1}\to \mathbb{S}^{n-1}$, \textit{e.g.}, $(x_1,x_2,\dots,x_n)\mapsto (x_1,x_2,\dots,-x_n)$, to change the orientation of $T$ (recall that $\mathrm{deg}(\rho)=-1$~\cite[Ch. 3]{ref:guillemin2010differential}), that is, we use $\rho\circ S$ instead of $S$. Thus, as $T$ can always be chosen to be orientation-preserving, $T$ (and equivalently $T^{-1}$) can assumed to be a \textit{stable homeomorphism}\footnote{The \textit{stable homeomorphism theorem}---stating that all orientation-preserving homeomorphisms of $\mathbb{R}^n$ are stable---connects to the \textit{annulus theorem} (which was a longstanding conjecture) via the work of Brown and Gluck~\cite{ref:BrownGluck} and has a rich history, with the key steps in the proof being provided by Kirby~\cite{ref:kirby1969stable} and Quinn~\cite{ref:quinn1982ends}. We point the reader to the survey of Edwards in the edited book by Gordon and Kirby on $4$-manifolds~\cite{ref:edwards1984solution}.} and hence $T$ is isotopic\footnote{Indeed, as also mentioned in~\cite{ref:grune1999asymptotic}, $T$ cannot (always) be a diffeomorphism on $\mathbb{R}^n$, for otherwise we constrain ourselves to smoothly conjugate systems.} to $\mathrm{id}_{\mathbb{R}^n}$, \textit{e.g.}, see~\cite{ref:kirby1969stable} and~\cite[Ch. 11]{ref:moise2013geometric}.  

(iii) First, we can always pick the homotopy $\widetilde{H}:[0,1]\times \mathbb{R}^n\to \mathbb{R}^n$ from $\mathrm{id}_{\mathbb{R}^n}$ to $T^{-1}$ such that $0$ is mapped to $0$ along the path in $\mathrm{Homeo}^+(\mathbb{R}^n;\mathbb{R}^n)$, \textit{e.g.}, for a homeomorphism $K:\mathbb{R}^n\to \mathbb{R}^n$ such that $K(0)\neq 0$, consider the map $x\mapsto L(x):=K(x)-K(0)$, with $y\mapsto L^{-1}(y)=K^{-1}(y+K(0))$. Now recall that $0\in \mathbb{R}^n$ is GAS under $\varphi^{t}(\cdot;-\nabla V)$, \textit{i.e.}, $V(\varphi^{t}(x;-\nabla V))-V(x)\leq -\int^{t}_0 W_V (\varphi^{\tau}(x;-\nabla V))\mathrm{d}\tau$ $\forall (t,x)\in \mathbb{R}_{>0}\times \mathbb{R}^n\setminus\{0\}$, with $x\mapsto W_V(x):=\|\nabla V(x)\|_2^2$. Since $\widetilde{H}(s,\cdot)$ is a homeomorphism with $\widetilde{H}(s,0)=0$, we also have 
\begin{equation*}
\label{equ:V1}
\begin{aligned}
    V(\varphi^{t}(\widetilde{H}(s,x);-\nabla V))-V(\widetilde{H}(s,x))\leq  
    -\int^{t}_0 W_V (\varphi^{\tau}(\widetilde{H}(s,x);-\nabla V))\mathrm{d}\tau, 
\end{aligned}
\end{equation*}
for any $(t,x)\in \mathbb{R}_{>0}\times \mathbb{R}^n\setminus\{0\}$. Next, define the semiflow $\widetilde{\varphi}(\cdot;s)$ through  
\begin{equation}
\label{equ:conj:semiflow}
    \widetilde{H}(s,\cdot)\circ \widetilde{\varphi}^{t}(\cdot;s) = \varphi^{t}(\cdot;-\nabla V)\circ \widetilde{H}(s,\cdot).  
\end{equation}
Note that $\widetilde{\varphi}^{t}(\cdot;0)=\varphi^{t}(\cdot;-\nabla V)$ whereas $\widetilde{\varphi}^{t}(\cdot;1)=T\circ \varphi^{t}(\cdot;-\nabla V)\circ T^{-1}$. Also note that continuity of $[0,1]\ni s\mapsto \widetilde{\varphi}^{t}(\cdot;s)$, in particular, continuity of $[0,1]\ni s\mapsto\widetilde{H}(s,\cdot)^{-1}$, follows from $\mathrm{Homeo}^+(\mathbb{R}^n;\mathbb{R}^n)$ being a \textit{topological group}~\cite{ref:arens1946topologies} (endowed with the compact-open topology), \textit{i.e.}, combine that the map $\mathrm{Homeo}^+(\mathbb{R}^n;\mathbb{R}^n)\ni h \mapsto h^{-1}$ is continuous with path-connectedness being preserved under a continuous map. Then, by~\eqref{equ:conj:semiflow}, we get
\begin{equation}
\label{equ:V2}
\begin{aligned}
 V(\widetilde{H}(s,\widetilde{\varphi}^{t}(x;s)))-V(\widetilde{H}(s,\widetilde{\varphi}^0(x;s)))\leq 
 -\int^{t}_0 W_V (\widetilde{H}(s,\widetilde{\varphi}^{\tau}(x;s)))\mathrm{d}\tau.
\end{aligned}
\end{equation}
Now, using~\eqref{equ:V2}, define $\widetilde{V}(\cdot;s)\in C^0(\mathbb{R}^n;\mathbb{R}_{\geq 0})$ through $x\mapsto\widetilde{V}(x;s):=V(\widetilde{H}(s,x))$ and similarly, define $x\mapsto\widetilde{W}_V(x;s):=W_V(\widetilde{H}(s,x))$ 
Hence, for any $s\in [0,1]$, we have $\widetilde{V}(\widetilde{\varphi}^{t}(x;s);s)-\widetilde{V}(\widetilde{\varphi}^0(x;s);s)\leq -\int^{t}_0 \widetilde{W}_V(\widetilde{\varphi}^{\tau}(x;s);s)\mathrm{d}\tau<0$, $\forall (t,x)\in \mathbb{R}_{>0}\times \mathbb{R}^n\setminus\{0\}$. Note that $x\mapsto\widetilde{V}(x;0)=V(x)$, $x\mapsto\widetilde{V}(x;1)=V(T^{-1}(x))=\gamma(\|x\|_2)=:V_\gamma(x)$ and $x\mapsto \widetilde{W}_V(x;1)=W_V(T^{-1}(x))$. Also, as $\widetilde{H}(s,\cdot)$ fixes $0$, $\widetilde{V}(\cdot;s)$ will always have the required compact sublevel sets.   

(iv) Although the path from $\mathrm{id}_{\mathbb{R}^n}$ to $T^{-1}$ is merely through $\mathrm{Homeo}^+(\mathbb{R}^n;\mathbb{R}^n)$, it is known that $T$ can be chosen to be diffeomorphic on $\mathbb{R}^n\setminus\{0\}$ for $n\neq 5$~\cite{ref:grune1999asymptotic} (the composition with $\rho$ does not change this)\footnote{Some comments are in place. Perelman's work came after \cite{ref:grune1999asymptotic}, resolving the case $n=4$ ($\mathbb{S}^3$). Despite some claims in the literature, to the best of our knowledge, the case $n=5$ ($\mathbb{S}^4$) is still open. As the generalized Poincar\'e conjecture is known to fail, in general, for $\mathsf{Diff}$ (e.g., due to the existence of \textit{exotic spheres}~\cite{ref:Milnor7}), we emphasize that the existence of these diffeomorphisms is due to Smale's \textit{h-cobordism theorem} (\textit{e.g.}, see \cite[Thm. 5.1]{ref:smale1962structure}). To be more explicit, we can construct a h-cobordism between $V^{-1}(c)$ and a levelset of $x\mapsto \|x\|_2^2$, \textit{e.g.}, by appealing to the same type of arguments as in \cite[Lem. III.4]{ref:jongeneelECC24}.}.
Hence, the vector field corresponding to $\widetilde{\varphi}(\cdot;1)$, denoted $\widetilde{F}$, is well-defined on $\mathbb{R}^n\setminus \{0\}$, \textit{i.e.}, 
\begin{equation}
\label{equ:DT:vecfield}
    \left.\frac{\mathrm{d}}{\mathrm{d}\tau}\widetilde{\varphi}^{\tau}(x;1)\right|_{\tau=0}=-\mathrm{D}T(T^{-1}(x))\nabla V(T^{-1}(x))=:\widetilde{F}(x),\,\, \forall x\in \mathbb{R}^n\setminus\{0\}.
\end{equation}
In general, $\widetilde{F}$ need not be continuous at $0$, this, because $\mathrm{D}T$ need not be continuous at $0$. However, $\gamma$ can always be chosen such that $T$ is $C^1$ on $\mathbb{R}^n$ with $\mathrm{D}T(0)=0$~\cite[Prop. 1]{ref:grune1999asymptotic}. Unfortunately, $V_{\gamma}$ is also by no means smooth at $0$ for any choice of $\gamma$. In fact, the appropriate $\gamma\in \mathcal{K}_{\infty}$ to guarantee that $T$ is sufficiently regular is of the following form. Let $\alpha\in \mathcal{K}$ be smooth, define $h(r):=\int^r_0 a(\tau)\mathrm{d}\tau$ and set $\gamma:=h^{-1}$. It follows that $\gamma'(r)=1/\alpha(\gamma(r))$ is smooth on $(0,+\infty)$, yet, $\lim_{r\to 0^+}\gamma'(r)=+\infty$, \textit{e.g.}, $r\mapsto \gamma(r)=r^{1/2}$. However, see that on $\mathbb{R}^n\setminus\{0\}$ we have $\langle \nabla V_{\gamma}(x),\widetilde{F}(x)\rangle \leq -\widetilde{W}_V(x)$, that is,
\begin{equation}
\label{equ:Vgamma:Lyap}
    \left\langle \gamma'(\|x\|_2)\frac{x}{\|x\|_2},\widetilde{F}(x)\right\rangle \leq - \widetilde{W}_V(x),\,\, \forall x\in \mathbb{R}^n\setminus\{0\}. 
\end{equation}
Thus, multiplying~\eqref{equ:Vgamma:Lyap} by the function $a\in C^0(\mathbb{R}^n;\mathbb{R}_{\geq 0})$, defined through $x\mapsto a(x):= \|x\|_2/\gamma'(\|x\|_2)$, yields that $V_q\in C^{\infty}(\mathbb{R}^n;\mathbb{R}_{\geq 0})$ and $W_q\in C^{\infty}(\mathbb{R}^n\setminus\{0\};\mathbb{R}_{>0})$, defined by $x\mapsto V_q(x):=\frac{1}{2}\|x\|_2^2$ and $x\mapsto W_q(x):=-a(x)\widetilde{W}_V(x)$, comprise a $C^{\infty}$ Lyapunov pair for $\widetilde{F}$, also recall Example~\ref{ex:irr:nabla:V}. Equivalently, one can observe that the level sets of $V_{\gamma}$ are standard spheres. 

Then, as in Step (i), construct a homotopy through non-vanishing vector fields on $\mathbb{R}^n\setminus\{0\}$, in this case from $\widetilde{F}$ to $-\nabla V_q=-\partial_x$, with $V_q$ being \textit{a} Lyapunov function asserting stability of $0$ throughout. 
Then, we show there is homotopy from $\varphi^t(\cdot;\widetilde{F})$ to $\varphi^{t}(\cdot;-\partial_x)$ along semiflows, such that $0$ is GAS. 
    
    (v) Since we have constructed the desired (i.e., preserving stability) homotopy between any time-$t$ map $\varphi^{t}(\cdot;\mathcal{F}[X])$ to $\varphi^{t}(\cdot;-\partial_x)$, we observe that all these maps are continuous in $t$, in fact, all steps show that joint continuity holds so that we can conclude on the existence of the homotopy from the semiflow $\varphi(\cdot;\mathcal{F}[X])$ to $\varphi(\cdot;-\partial_x)$.  
\end{proof}

As Proposition~\ref{prop:GAS:homotopy:semiflow} is about \textit{existence}, we provide some examples below. 
\begin{example}[Homotopic vector fields, preserving stability]
\label{ex:examples:homotopy:semiflow}
\begin{upshape}
     It is known that the polynomial ODE $\dot{x}=X_1(x):= (-x_1+x_1x_2,-x_2)$ does not admit a polynomial Lyapunov function asserting $0$ is GAS~\cite{ref:ahmadi2011globally}. A non-polynomial Lyapunov function, however, is given by $(x_1,x_2)\mapsto V_1(x_1,x_2):=\tfrac{1}{2}\log(1+x_1^2)+\tfrac{1}{2}x_2^2$. To construct a homotopy from $\varphi(\cdot;X_1)$ to $\varphi(\cdot;-\partial_x)$, first construct the straight-line homotopy from $X_1$ to $-\nabla V_1$, with stability asserted throughout by $V_1$. Next, define $V(\cdot;s):=(1-s)V_1+sV_q$, for $s\in [0,1]$ and $x \mapsto V_q(x):=\tfrac{1}{2}\|x\|_2^2$. As for any $s\in [0,1]$, $\nabla V(x;s)=0\iff x=0$, $[0,1]\ni s \mapsto -\nabla V(\cdot;s)$ parametrizes a homotopy from $-\nabla V_1$ to $-\partial_x$, along vector fields that render $0$ GAS, asserted by $V(\cdot;s)$. Although $X_1$ does not admit a convex polynomial Lyapunov function, the negative gradient flow of the Lyapunov function corresponding to $X_1$ does \textit{cf}.~\cite{ref:JongeneelSchwan2024TAC}. 
    \exampleEnd
    \end{upshape}
\end{example}

    \begin{figure}
        \centering
        \includegraphics[trim={1cm 3cm 1cm 3cm},clip,scale=0.125]{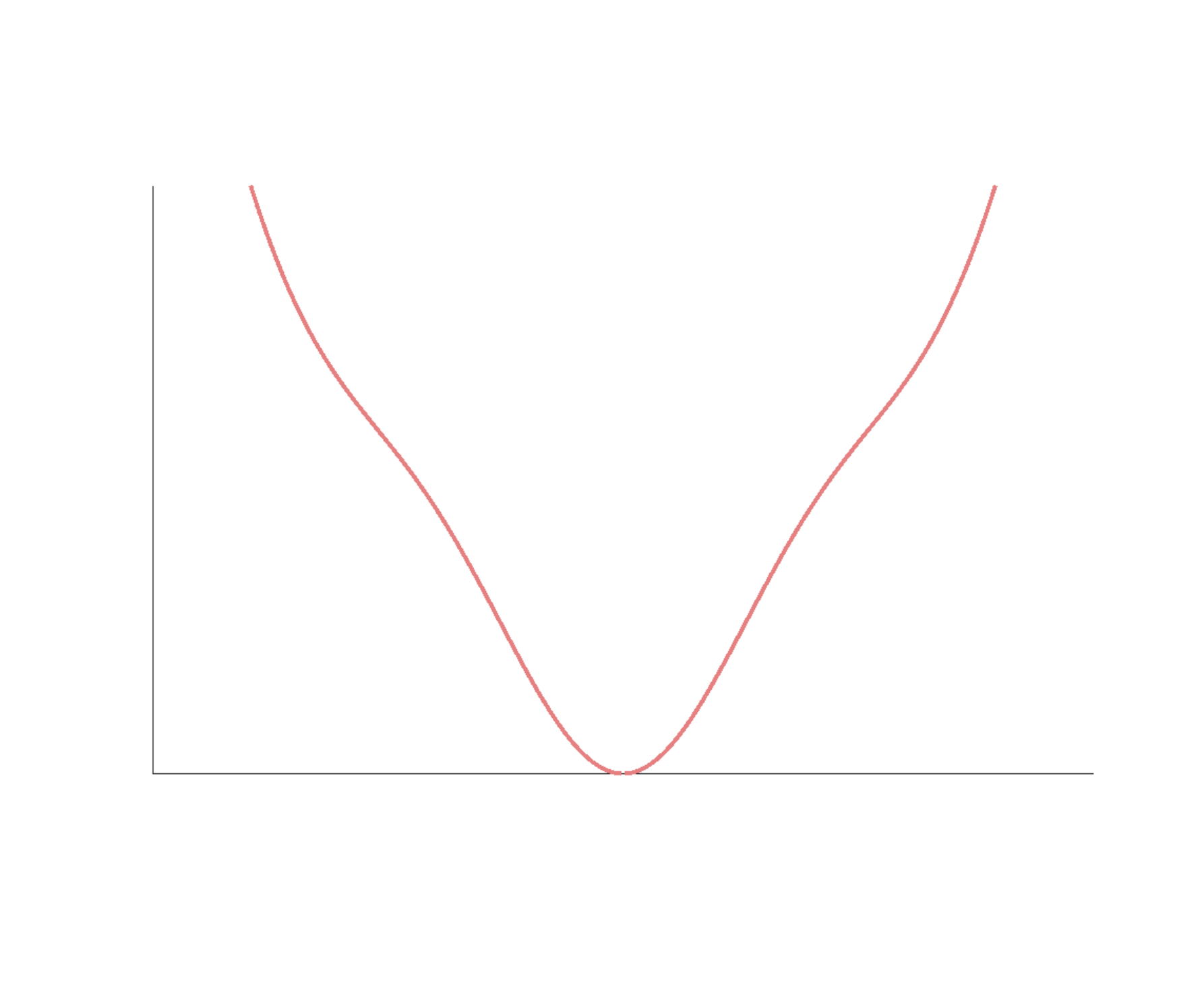}\hspace{-0.25cm}
        \includegraphics[trim={1cm 3cm 1cm 3cm},clip,scale=0.125]{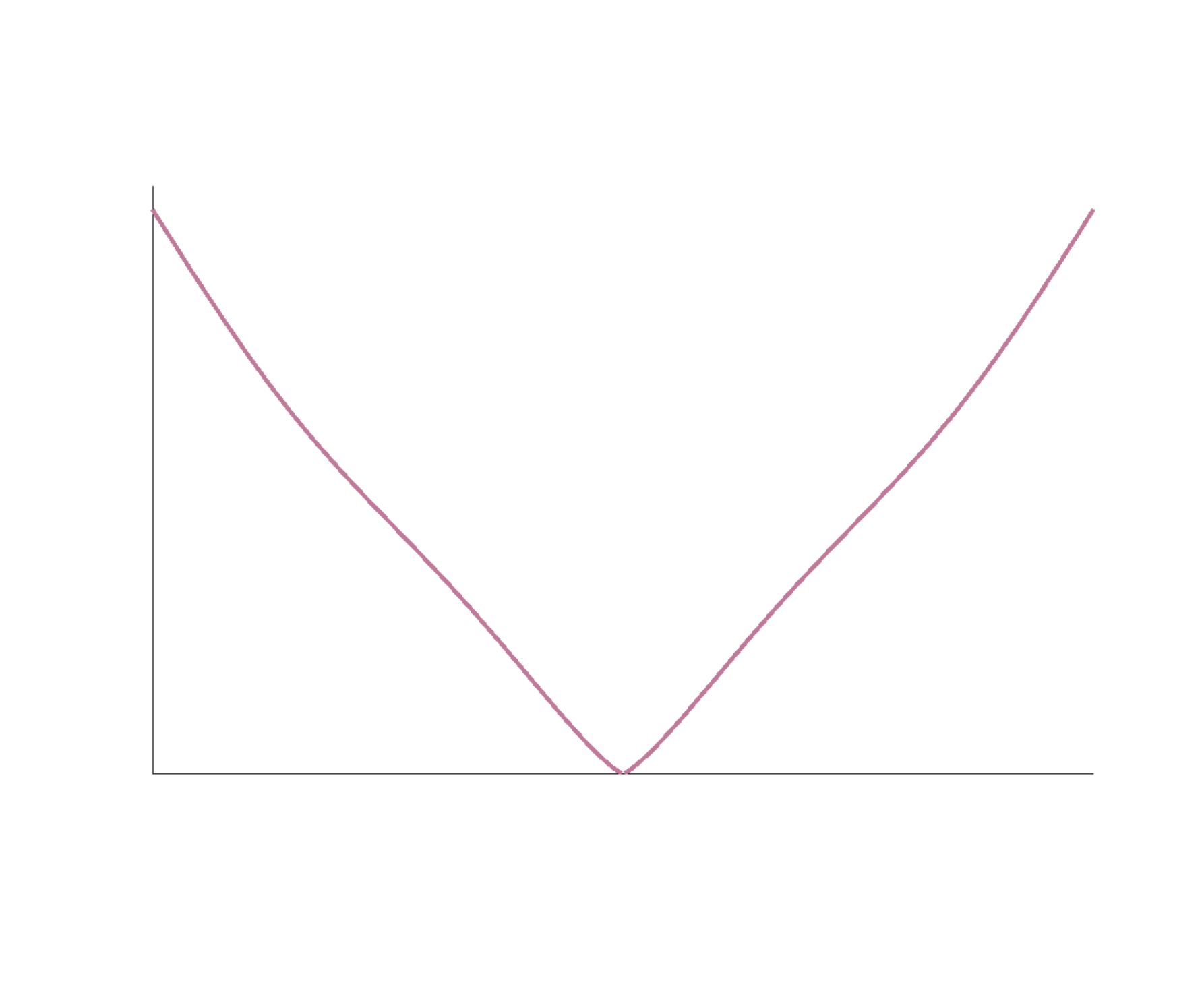}\hspace{-0.25cm}
                \includegraphics[trim={1cm 3cm 1cm 3cm},clip,scale=0.125]{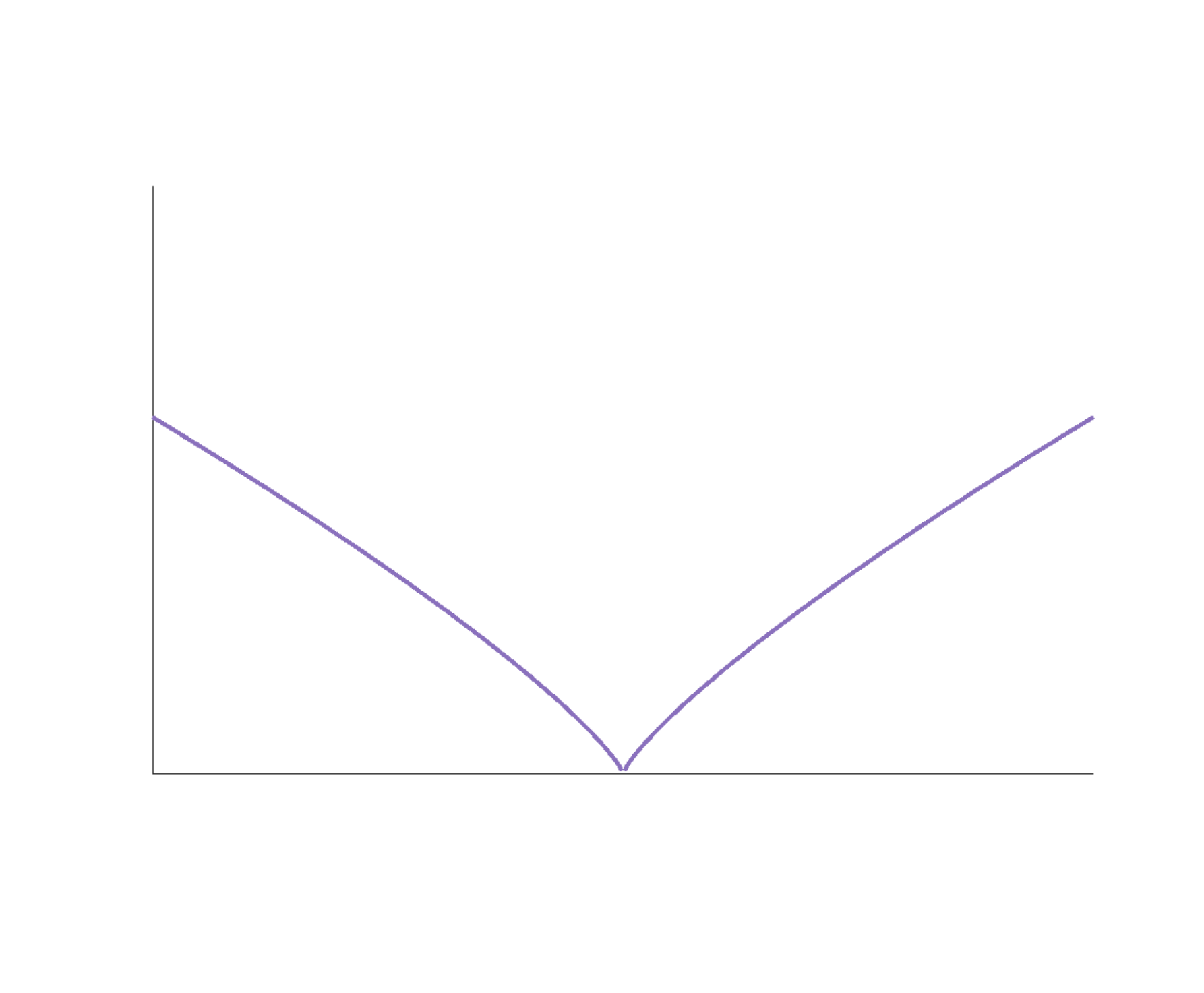}\hspace{-0.25cm}
         \includegraphics[trim={1cm 3cm 1cm 3cm},clip,scale=0.125]{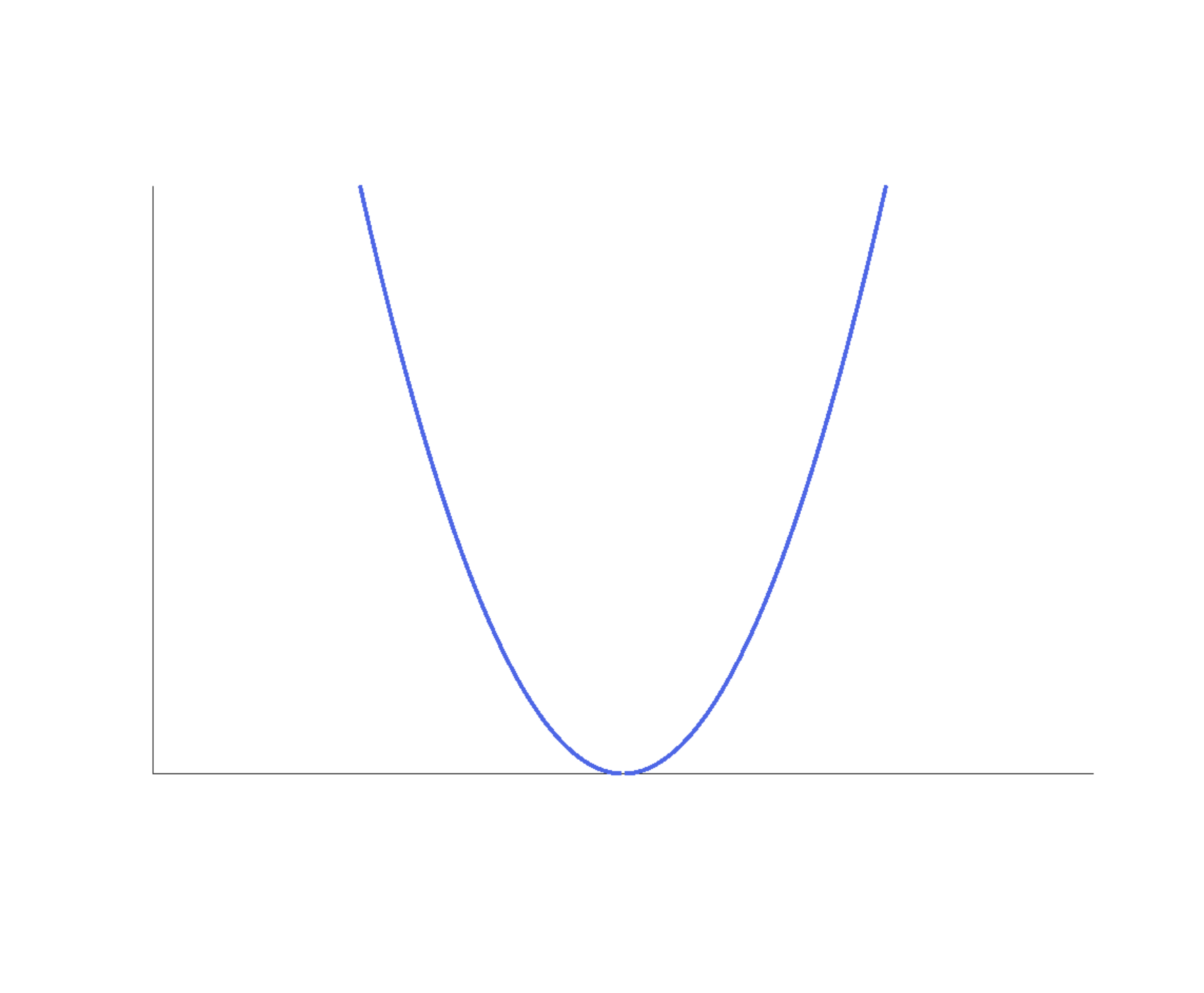}
        \caption{Example~\ref{ex:invextoconvex}: on the left, the graph of $v_i$ around $x=0$; and on the right, the graph of $v_q$ around $x=0$. In between, steps of the homotopy that connects $v_i$ to $v_q$, while preserving that $x=0$ is the global minimizer.}
        \label{fig:invextoconvex}
    \end{figure}
{The purpose of the following example is to illustrate the construction of the homotopy that appears in the proof of Proposition~\ref{prop:GAS:homotopy:semiflow}.}
\begin{example}[Homotopy from invexity to convexity]
\label{ex:invextoconvex}
\begin{upshape}
Consider a coercive, invex function (\textit{i.e.}, every critical point is a global minimizer) on $\mathbb{R}$ defined by $x\mapsto v_i(x):=\tfrac{1}{2}x^2+\tfrac{3}{2}\sin(x)^2$. Along the lines of the proof of Proposition~\ref{prop:GAS:homotopy:semiflow}, we can find a map $T$ such that $v_i(T^{-1}(x))=\gamma(|x|):=\sqrt{|x|}$ ({see that $\gamma(s)/\gamma'(s)\geq s$ for all $s\geq 0$}), and a path from $T$ to $\mathrm{id}_{\mathbb{R}}$, as given by
\begin{equation*}
    [0,1]\times \mathbb{R}\ni(s,x)\mapsto H_i(s,x) := 
        \mathrm{sgn}(x)\left(\frac{1+s}{2}x^2+(1-s)\frac{3}{2}\sin(x)^2\right)^{2-(3/2)s}, 
\end{equation*}
with $H_i(0,\cdot)=T$ and $H_i(1,\cdot)=\mathrm{id}_{\mathbb{R}}$ (\textit{e.g.}, $v_i=\sqrt{|H_i(0,\cdot)|}$). Next, consider the homotopy $(s',x)\mapsto x^{1/2+(3/2)s'}$ to construct the path from $v_i$ to $x\mapsto v_q(x):= x^2$, along continuous functions such that $0$ is the global minimizer throughout\footnote{For a simulation of this homotopy, see~\url{wjongeneel.nl/figinvex.gif}.}, see Figure~\ref{fig:invextoconvex}. 
For functions as simple as $v_i$ one can find simpler homotopies ({\textit{e.g.}, a straight-line}), however, to the best of our knowledge, being able to guarantee the mere existence of such a homotopy is new. Proposition~\ref{prop:GAS:homotopy:semiflow} provides us with the existence of a homotopy from a smooth Lyapunov function $V$ to $V_q$, along \textit{continuous}\footnote{It is not evident, and currently unknown, whether smoothness can be preserved throughout the homotopy, see Step (iii) of the proof of Proposition~\ref{prop:GAS:homotopy:semiflow}.} Lyapunov functions that assert $0\in \mathbb{R}^n$ is GAS along the homotopy. Differently put, we can find a homotopy from a coercive, invex function, to a convex function, such that along the homotopy the minimizer is preserved. This might be of independent interest.    
\exampleEnd
\end{upshape}
\end{example}

As we allow for a class of discontinuous vector fields, Proposition~\ref{prop:GAS:homotopy:semiflow} allows for an extension of the standard Hopf index, this has been pioneered by Gottlieb, \textit{e.g.}, see~\cite{ref:gottlieb1995index}. See also~\cite{ref:casagrande2008conley} for a Conley index applicable to discontinuous vector fields and see~\cite{ref:Kvalheim2021PHhybrid} for a hybrid Poincar{\'e}-Hopf theorem.

We point out that one could omit Step (i) of the proof of Proposition~\ref{prop:GAS:homotopy:semiflow}, yet, as is also shown in Example~\ref{ex:examples:homotopy:semiflow}, it is typically convenient to pass through a gradient flow. We also remark that the origin plays no particular role in Proposition~\ref{prop:GAS:homotopy:semiflow}, as it should. In fact, if $X$ is such as in Proposition~\ref{prop:GAS:homotopy:semiflow}, yet, the equilibrium point is now arbitrary, we can still construct a homotopy between $\varphi(\cdot;\mathcal{F}[X])$ and $\varphi(\cdot;-\partial_x)$ such that along the homotopy \textit{some} point is GAS. For instance, the following family of time-$t$ maps parametrizes a homotopy between the flows corresponding to the ODEs $\dot{x}=-x$ and $\dot{x}=-(x-\bar{x})$: $(t,x)\mapsto \varphi^{t}(x;s) = e^{-t}x + s(I_n-e^{-t}I_n)\bar{x}$, for $s\in [0,1]$. Indeed, $s\bar{x}$ is GAS along the homotopy (\textit{e.g.}, consider the Lyapunov function $x\mapsto V(x;s):=\frac{1}{2}(x-s\bar{x})^2$). 

Next, we generalize Proposition~\ref{prop:GAS:homotopy:semiflow} to smooth manifolds, which is almost immediate as the GAS property heavily restricts the class of manifolds, \textit{e.g.}, see~\cite[Ch. V.3]{ref:bhatia1970stability}. We assume that our manifolds are Hausdorff and second countable. Regarding our vector fields, we assume simply the following.

\begin{assumption}[Vector field regularity on $M^n$]
\label{ass:reg:X:M}
Our vector fields are locally Lipschitz on $M^n$. 
\end{assumption}

Indeed, one can work with less regular vector fields and fully generalize Proposition~\ref{prop:GAS:homotopy:semiflow} to manifolds, \textit{i.e.}, via~\cite[Cor. 13]{ref:mayhew2011topological}. We refrain from introducing further technicalities and work with Assumption~\ref{ass:reg:X:M}. 

\begin{theorem}[Stability preserving semiflows]
\label{thm:GAS:homotopy:semiflow:M}
Let $M^n$ be a smooth manifold, for $n\neq 5$, and let $p^{\star}\in M^n$ be GAS under both the vector fields $X$ and $Y$, with both $X$ and $Y$ satisfying Assumption~\ref{ass:reg:X:M}. Then, the flow $\varphi(\cdot;X)$ is homotopic to $\varphi(\cdot;Y)$, along semiflows that preserve $p^{\star}$ being GAS. 
\end{theorem}
\begin{proof}
First, since $p^{\star}$ is GAS under $X$, there is $V\in C^{\infty}(M^n;\mathbb{R}_{\geq 0})$ such that $X(V)=L_X V<0$ on $M^n\setminus \{p^{\star}\}$~\cite{ref:fathi2019smoothing}. Now fix some Riemannian metric, which exists~\cite[Ch. 13]{Lee2}, and define the Riemannian gradient $\mathrm{grad}\,V$. It follows that $-\mathrm{grad}\,V(V)<0$ on $M^n\setminus\{p^{\star}\}$. Then, for any $t\in \mathbb{R}$, the corresponding time-$t$ map $p\mapsto \varphi^t(p;-\mathrm{grad}\,V)$ is a diffeomorphism. This implies there is a diffeomorphism $\psi:M^n\to \mathbb{R}^n$~\cite[Thm. 2.2]{ref:wilson1967structure}. Next, define the diffeomorphism $\widetilde{\psi}$ by ${M}^n\ni p\mapsto\widetilde{\psi}(p):=\psi(p)-\psi(p^{\star})\in \mathbb{R}^n$, such that $\widetilde{\psi}(p^{\star})=0$. Now set $x:=\widetilde{\psi}(p)$ and consider
    \begin{align*}
        \dot{x} = -\mathrm{D}\widetilde{\psi}(\widetilde{\psi}^{-1}(x))X(\widetilde{\psi}^{-1}(x))=:X_0(x). 
    \end{align*}
    Indeed, the vector field $X_0$ meets the criteria of Proposition~\ref{prop:GAS:homotopy:semiflow}. Then, observe that  $\varphi^{t}(\cdot;X) = \widetilde{\psi}^{-1}\circ\varphi^{t}(\cdot;X_0)\circ \widetilde{\psi}$, which implies that $\varphi^{t}(\cdot;X)$ can be homotoped to $\widetilde{\psi}^{-1}\circ\varphi^{t}(\cdot;-\partial_x)\circ \widetilde{\psi}$ such that $p^{\star}$ remains GAS along the homotopy. We can do the exact same for the vector field $Y$. Then, since $\widetilde{\psi}$ works for both $X$ and $Y$ we can conclude by transitivity.  
\end{proof}

    \begin{figure}
        \centering
        \includegraphics[trim={3cm 3cm 3cm 3cm},clip,scale=0.15]{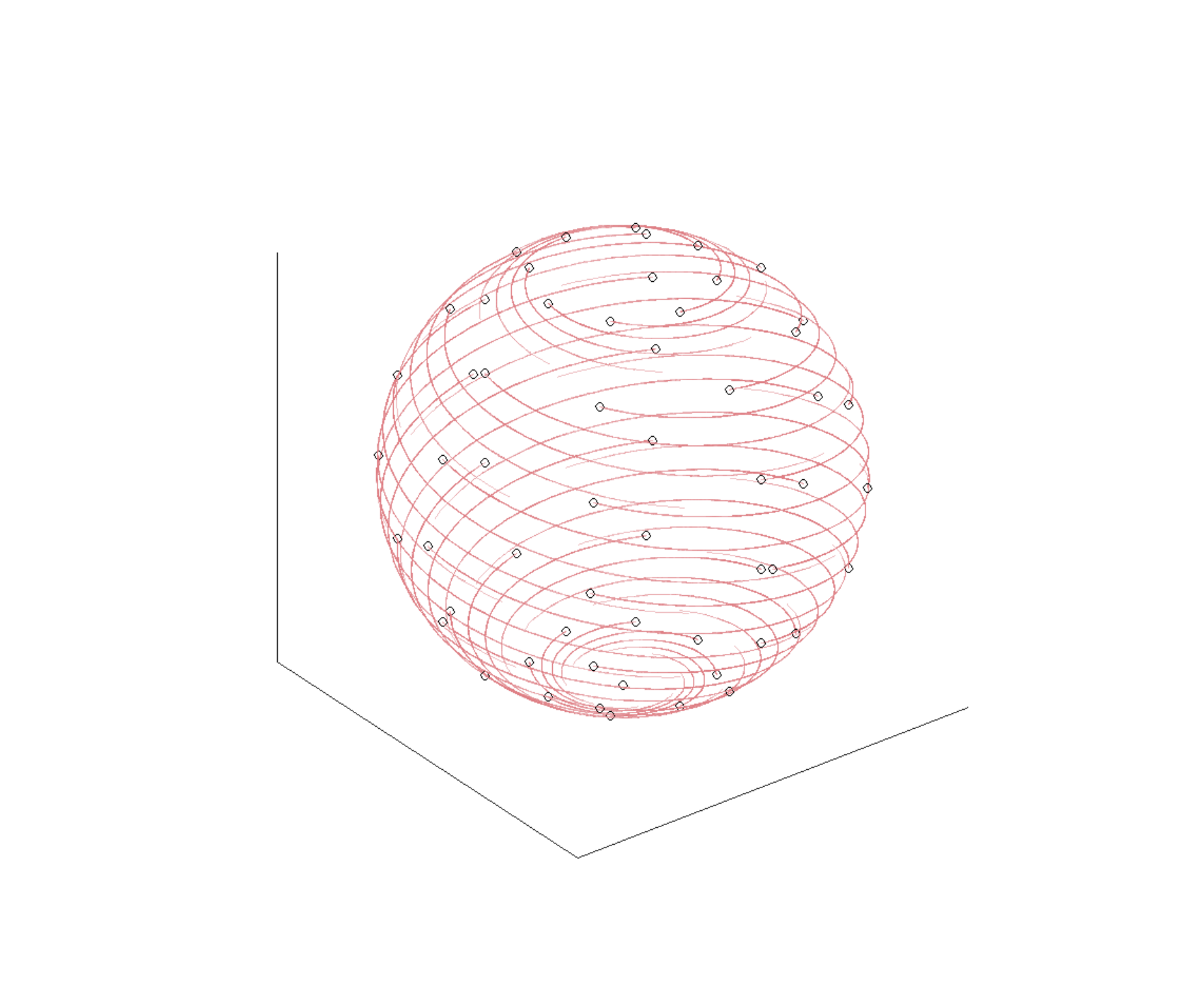}\hspace{-0.5cm}
        \includegraphics[trim={3cm 3cm 3cm 3cm},clip,scale=0.15]{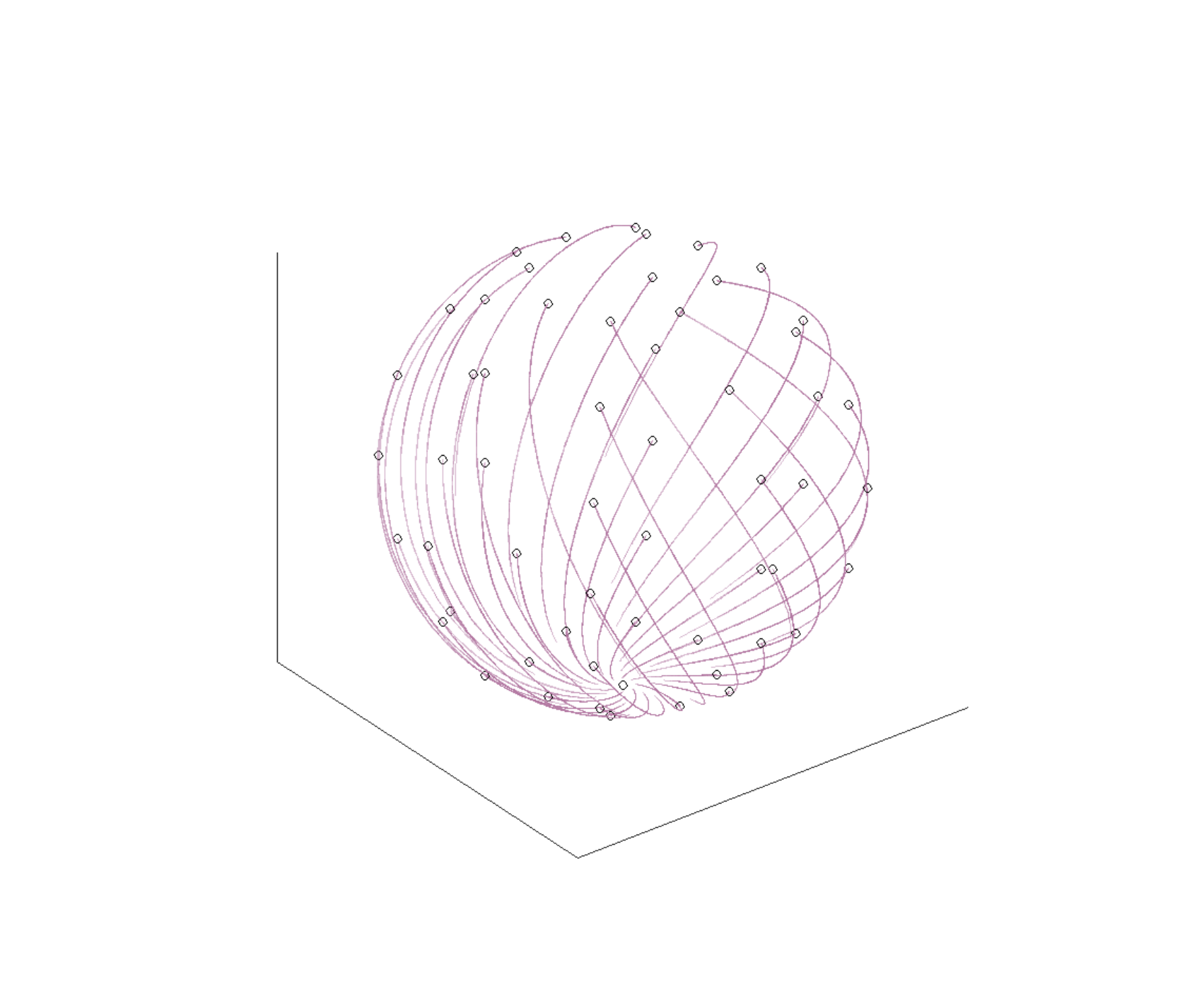}\hspace{-0.5cm}
                \includegraphics[trim={3cm 3cm 3cm 3cm},clip,scale=0.15]{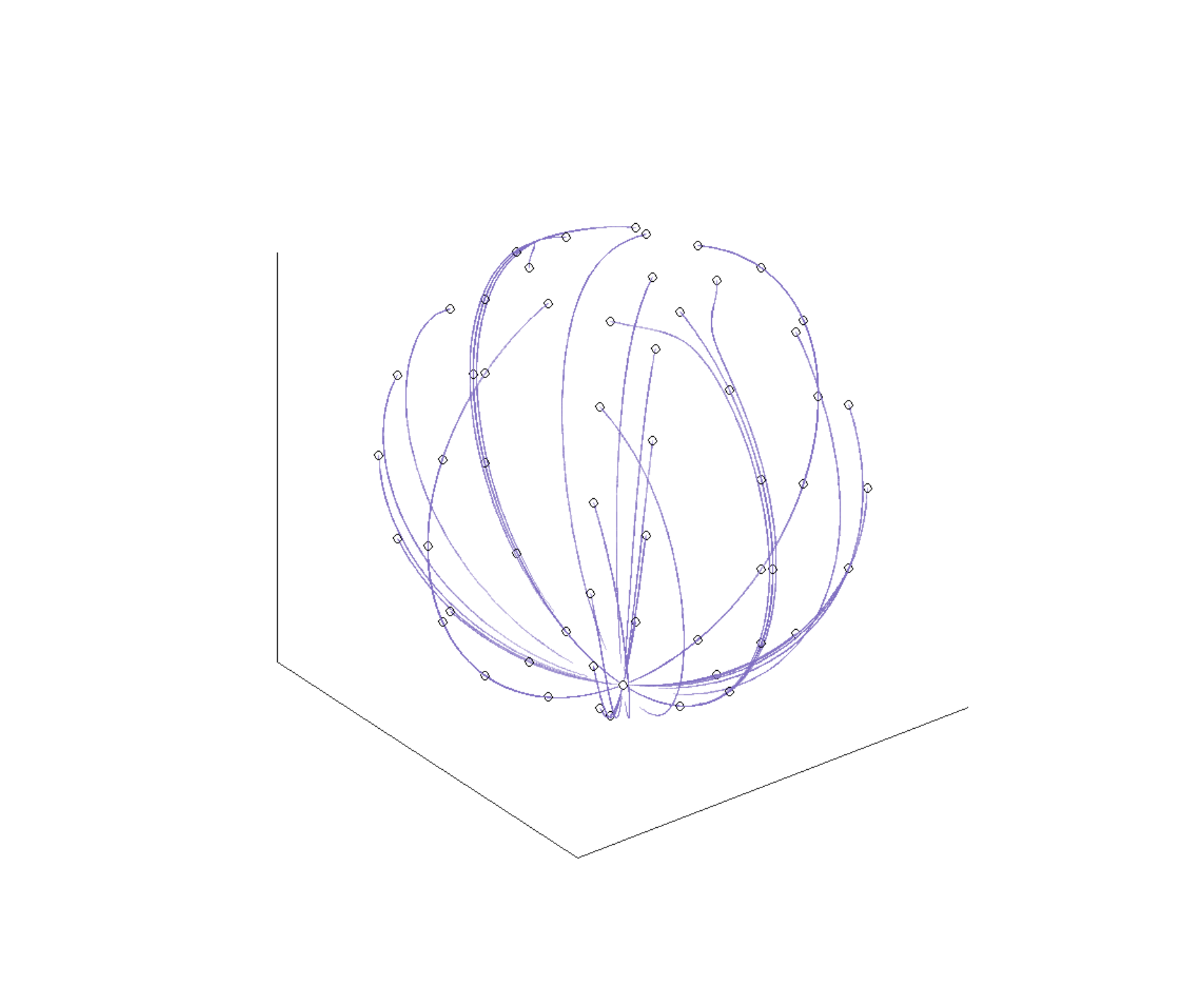}\hspace{-0.5cm}
         \includegraphics[trim={3cm 3cm 3cm 3cm},clip,scale=0.15]{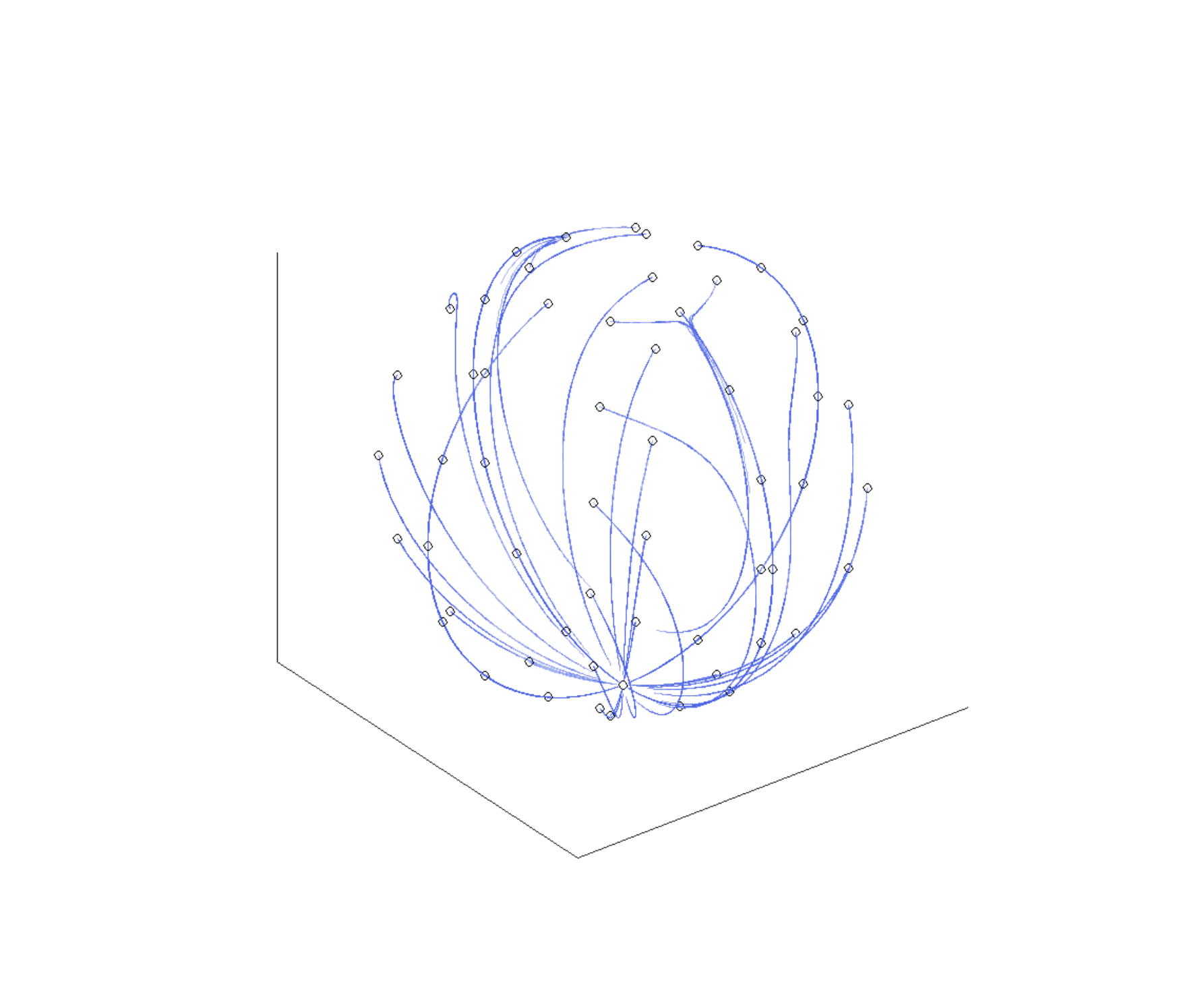}
        \caption{Example~\ref{ex:StereoS2}: on the left, some flow lines under $(\Pi_N^{-1})_* X$; and on the right, some flow lines under $(\Pi_N^{-1})_* Y$. In between, two steps of the homotopy that connects $(\Pi_N^{-1})_* X$ to $(\Pi_N^{-1})_* Y$, through vector fields that render $S$ GAS on $\mathbb{S}^2\setminus \{N\}$.}
        \label{fig:StereoS2}
    \end{figure}
\begin{example}[Equilibria on $\mathbb{S}^2$]
\label{ex:StereoS2}
\begin{upshape}
To exemplify Theorem~\ref{thm:GAS:homotopy:semiflow:M}, one might think of rendering the South Pole $S$ an attractor on $\mathbb{S}^2\setminus \{N\}$, for $N$ the North Pole. To that end, consider the vector fields $X$ and $Y$ on $\mathbb{R}^2$, defined through
\begin{align*}
X(x) &:= (-0.1x_1-x_2)\partial_{x_1}+(x_1-0.1x_2)\partial_{x_2},\\ Y(x) &:= (-x_1-x_1x_2^2)\partial_{x_1}+(-x_2+x_1^2x_2)\partial_{x_2}.
\end{align*}
The origin is GAS under both $X$ and $Y$ (\textit{e.g.}, consider the canonical quadratic Lyapunov function). Let $\Pi_{N}$ be the stereographic projection from $\mathbb{S}^2\setminus \{N\}$ to $\mathbb{R}^2$. Then, to transform $X$ and $Y$ to vector fields on $\mathbb{S}^2\setminus\{N\}$, we construct the \textit{pushforwards} $(\Pi_N^{-1})_* X$ and $(\Pi_N^{-1})_* Y$, see Figure~\ref{fig:StereoS2}. Exploiting this structure and our previous work~\cite{ref:JongeneelSchwan2024TAC}, we can construct an explicit homotopy between these two vector fields that preserves stability of $S$ on $\mathbb{S}^2\setminus
\{N\}$\footnote{For a numerical simulation of the homotopy, see \url{wjongeneel.nl/figStereoS2.gif}. Note that the homotopy is through the canonical vector field indeed.}. 
\exampleEnd
\end{upshape}
\end{example}

\begin{remark}[Weaker notions of stability]
\begin{upshape}
    In general, one cannot relax GAS to mere stability. A reason being that the (Hopf) index of Lyapunov stable equilibria is not fixed~\cite[Sec. 52]{ref:krasnosel1984geometrical}, yet, this index is a homotopy invariant (\textit{e.g.}, for GAS the index is fixed). In Section~\ref{sec:ISS}, we do elaborate on ISS. \remarkEnd
    \end{upshape} 
\end{remark}

{
\section{Input-to-State Stability}
\label{sec:ISS}
As in~\cite{ref:grune1999asymptotic},
we can also study ``\textit{disturbed}'' systems of the form
\begin{equation}
\label{equ:system}
    \dot{x}=f(x,d),
\end{equation}
where $f:\mathbb{R}^n\times D\to \mathbb{R}^n$, with $D\subseteq \mathbb{R}^m$, is continuous, and locally Lipschitz on $\mathbb{R}^n\setminus \{0\}\times D$. 
We let $\mathscr{D}_I$ denote the set of measurable, locally essentially bounded functions $\mathbb{R}\supseteq I\ni t\mapsto d(t)\in D$, with $d(\cdot)\in \mathscr{D}$ overloading notation, indicating any function of appropriate length. Whenever relevant, we do assume that our solutions are forward complete (\textit{e.g.}, we can appeal to \cite[Cor. 2.11]{ref:angeli1999forward}), we denote them using semiflow notation as $\varphi^t(x_0,d(\cdot))$. 

Leveraging intuition from linear systems theory, input-to-state stability was derived as a stability notion invariant under coordinate transformations (homeomorphisms that fix $0$) \cite{ref:sontagbasicconcepts2008input}. In particular, a system~\eqref{equ:system} is said to be {\textit{Input-to-State Stable}} (ISS) when there are $\beta\in \mathcal{KL}$ and $\gamma\in \mathcal{K}_{\infty}$ such that
\begin{equation*}
\label{equ:ISS:Linf}
\|\varphi^t(x_0,d(\cdot))\|\leq \max\left\{\beta(\|x_0\|,t), \gamma(\|d\|_{\infty})\right\}
\end{equation*}
for all $t\geq 0$, $x_0\in \mathbb{R}^n$ and any $d(\cdot)\in \mathscr{D}$. For the corresponding Lyapunov theory, see~\cite{ref:sontag1995characterizations}. Constraining the transient, a system~\eqref{equ:system} is said to be {\textit{Input-to-State Exponentially Stable}} (ISES) when there are constants $M,a\in \mathbb{R}_{>0}$ and a $\alpha\in \mathcal{K}_{\infty}$ such that
\begin{equation}
\label{equ:ISES:max}
    \|\varphi^t(x_0,d(\cdot))\|\leq \max\left\{Me^{-a t}\|x_0\|, \alpha(\|d\|_{\infty})\right\}
\end{equation}
for all $t\geq 0$, $x_0\in \mathbb{R}^n$ and any $d(\cdot)\in \mathscr{D}$. It can be shown \cite[Thm. 3]{ref:grune1999asymptotic}, that there is always a coordinate transformation that brings an ISS system into one that is ISES with normalized constants, that is, $M=a=1$.  

Inspired by the above, it seems plausible that $\dot{x}=-x+d$ fulfils the role of ``\textit{canonical ISS system}'' on $\mathbb{R}^n$. For instance, recall that the solution satisfies $\varphi^t(x_0,d(\cdot)) = e^{-t}x_0 + \int^t_0 e^{-(t-\tau)}d(\tau) \mathrm{d}\tau$ for any $t\geq 0$, $x_0\in \mathbb{R}^n$ and $d(\cdot)\in \mathscr{D}$. 
Therefore, $\|\varphi^t(x_0,d(\cdot))\|\leq \max\{2|e^{-t}|\|x_0\|,2 \|d\|_{\infty}\}$.
In that sense, $\dot{x}=-x+d$ is canonically ISS due to the simple comparison functions we can employ. Moreover, set $x\mapsto S(x):=\frac{1}{2}\|x\|_2^2$, it follows that $\langle \partial_x S(x),-x+d\rangle\leq \tfrac{1}{2}(\|d\|_2^2-\|x\|_2^2)$ for all $x,d\in \mathbb{R}^n$
and thus
\begin{equation}
\label{equ:canonical:L2}
    \int^t_0 \|\varphi^{\tau}(x_0,d(\cdot))\|_2^2\mathrm{d}\tau \leq \|x_0\|_2^2 + \int^t_0 \|d(\tau)\|_2^2 \mathrm{d}\tau
\end{equation}
for all $t\geq 0$, $x_0\in \mathbb{R}^n$ and any $d(\cdot)\in \mathscr{D}$, \textit{e.g.}, see \cite[Ch. 8]{ref:AvdSL2}. However, this means that $\dot{x}=-x+d$ can be understood, in particular, as a canonical system with a finite (unitary) linear $L_2$ gain. Indeed, \cite[Thm. 4]{ref:grune1999asymptotic} shows that an ISES system can be transformed into a system with such an $L_2$ gain.

Here, we focus on homotopies through coordinate transformations of $x$ and $d$, for otherwise we could ``remove'' the disturbance, \textit{e.g.}, consider $[0,1]\ni \theta \mapsto f(x,\theta d)$. Now, to exemplify why the existence of a homotopy from any ISS system, through coordinate transformation, to $\dot{x}=-x+d$ is too strong, consider the following example. Example~\ref{ex:no:homotopy} implies in particular that a na\"ive generalization of Section~\ref{sec:main} to ISS is impossible.
\begin{example}[No canonical system]
\label{ex:no:homotopy}
\begin{upshape}
    Consider the linear ISS system
    \begin{equation}
    \label{equ:system:ex:2}
        \begin{pmatrix}
\dot{x}_1\\ \dot{x}_2
        \end{pmatrix} = 
        \begin{pmatrix}
            -1 & 0\\
            0 & -1
        \end{pmatrix}\begin{pmatrix}
{x}_1\\ {x}_2
        \end{pmatrix} +
        \begin{pmatrix}
            1 & 0\\
            0 & -1
        \end{pmatrix}\begin{pmatrix}
{d}_1\\ {d}_2
        \end{pmatrix},
    \end{equation}
or in short: $\dot{x}=Ax+R(d)$. As $R\notin \mathrm{Homeo}^+(\mathbb{R}^2;\mathbb{R}^2)$, there is no continuous path from $\mathrm{id}_{\mathbb{R}^2}$ to $R$ (or to $R^{-1}$ for that matter), through coordinate transformations. However, if we pick $x\mapsto S(x):=\tfrac{1}{2}\|x\|_2^2$ as our storage function then we again find that the system~\eqref{equ:system:ex:2} satisfies $\langle\partial_x S(x), Ax+R(d)\rangle\leq \tfrac{1}{2}(\|d\|_2^2-\|x\|_2^2)$ and thus, \eqref{equ:system:ex:2} satisfies the canonical $L_2$ gain bound~\eqref{equ:canonical:L2}. 
\exampleEnd
\end{upshape}
\end{example}

The way to interpret the upcoming proposition is as follows. Given a solution under~\eqref{equ:system}, assumed to be ISS. This solution is understood as fixed data and now someone gradually applies a change of coordinates to both $x$ and $d$. It turns out that by doing so, one can always transform this solution data into data satisfying~\eqref{equ:canonical:L2}, thus, into a system \textit{like} $\dot{x}=-x+d$. Indeed, it follows directly that ISS is preserved throughout, \textit{e.g.}, see~\cite{ref:kellett2015input} for more on ISS and coordinate transformations. However, comparing again to Section~\ref{sec:main}, we cannot readily say more beyond \textit{paths of coordinate transformations}, \textit{e.g.}, first one should topologize $\mathscr{D}$ to study continuity of $d(\cdot)\mapsto \varphi^t(x_0,d(\cdot))$. In this work we refrain from completing this study as it does not add to the central message while it requires significant technical machinery, however, without doing so we cannot formally discuss homotopies of \textit{solutions}\footnote{This is also why ``\textit{and ISS equals finite energy}'' is not part of our title \textit{cf}.~\cite{ref:grune1999asymptotic}.}. 

\begin{proposition}[Homotopy from ISS into a canonical $L_2$ gain]
\label{prop:homotopy:ISS:L2}
Let $n\neq 5$ and suppose that the system~\eqref{equ:system} is ISS, then we can always find a continuous path of coordinate transformations of both $x$ and $d$, transforming the solutions to \eqref{equ:system} into solutions of a system satisfying the canonical bound~\eqref{equ:canonical:L2}. 
\end{proposition}
\begin{proof}

First, by \cite[Thm. 3]{ref:grune1999asymptotic} and Proposition~\ref{prop:GAS:homotopy:semiflow}, if \eqref{equ:system} is ISS, we can always find a continuous path of coordinate transformations of $x$, transforming the solutions to \eqref{equ:system} into solutions of a system that is ISES with normalized constants, while still admiting a representation of the form \eqref{equ:system}.

Second, as we may suppose that \eqref{equ:system} is ISES with normalized constants, then by \cite[Thm. 4]{ref:grune1999asymptotic} and arguments akin to Proposition~\ref{prop:GAS:homotopy:semiflow}, we can always find a continuous path of coordinate transformations of $d$, transforming the solutions to \eqref{equ:system} into solutions satisfying the canonical $L_2$ bound~\eqref{equ:canonical:L2}. To be somewhat explicit, 
    since~\eqref{equ:system} is ISES, then by~\eqref{equ:ISES:max}, 
    if 
    \begin{equation}
    \label{equ:xsupd}
        \|x\|_2\geq e^s \alpha \left(\textstyle\sup_{\tau\in [0,s]}\|d(\tau)\|_2\right)
    \end{equation}
    then $\|\varphi^t(x,d(\cdot))\|_2\leq e^{-t}\|x\|_2$ for all $t\in [0,s]$. Let $V:\mathbb{R}^n\to \mathbb{R}_{\geq 0}$ be defined through $x\mapsto V(x):=\|x\|_2^2$, then if \eqref{equ:xsupd} holds we readily get
    \begin{equation}
    \label{equ:ode:V}
        V(\varphi^t(x,d(\cdot)))\leq e^{-2t}V(x)
    \end{equation}
    for all $t\in [0,s]$. This must be true in particular for constant disturbances, say $t\mapsto d(t):=d'\in D$ for all $t\geq 0$, thus, by subtracting $V(\varphi^0(x,d'))=e^{-2\cdot 0}\|x\|_2^2=V(x)$ on both sides of \eqref{equ:ode:V}, dividing by $t$ and considering $t\to 0^+$ we get that $\|x\|_2 \geq \alpha(\|d'\|_2)$ implies that $L_{f(\cdot,d')} V(x):=\langle \partial_xV(x), f(x,d')\rangle \leq -2 V(x) \leq - V(x)$. Now, akin to \cite[Rem. 2.4]{ref:sontag1995characterizations}, define
    \begin{equation*}
        \widetilde{\alpha}(r):= \sup_{\|\xi\|_2\leq \alpha(r)\,\|u\|_2\leq r}2\cdot\langle f(\xi,u),\xi\rangle+ V(\xi).
    \end{equation*}
    Evidently, as $\widetilde{\alpha}\in \mathcal{K}_{\infty}$, then for $\|x\|_2\geq \alpha(\|d'\|_2)$, we have that
    \begin{equation}
        \label{equ:ISS:ineq}
        L_{f(\cdot,d')}V(x)\leq -V(x) + \widetilde{\alpha}(\|d'\|_2).
    \end{equation}
    On the other hand, for $\|x\|_2\leq \alpha(\|d'\|_2)$, \eqref{equ:ISS:ineq} also holds true. As $d'\in D$ was arbitrary, \eqref{equ:ISS:ineq} holds true for any $d'\in D$. Now, let $\widehat{f}(\cdot,v):=f(\cdot,R^{-1}(v))$ and see that if, towards~\eqref{equ:canonical:L2}, the goal is to get $L_{\widehat{f}(\cdot,v)}V(x)\leq - V(x) + \|v\|_2^2$, then we could select $\mathbb{R}^m\ni z\mapsto R(z):= (\widetilde{\alpha}(\|z\|_2))^{1/2}(z/\|z\|_2)$, with $R(0):=0$. This map is $C^{\infty}$ on $\mathbb{R}^m\setminus \{0\}$ and continuous at $0$ as $\widetilde{\alpha}\in \mathcal{K}_{\infty}$. 
    For $R$ to be orientation preserving, we can again introduce a map $\rho:\mathbb{S}^{m-1}\to \mathbb{S}^{m-1}$ of appropriate degree, like in Step (ii) of the proof of Proposition~\ref{prop:GAS:homotopy:semiflow}, \textit{e.g.}, $z \mapsto (\widetilde{\alpha}(\|z\|_2))^{1/2}\rho\left(z/\|z\|_2 \right).$
\end{proof}
 }

\section{A view from optimal transport}
\label{sec:OT}


Suppose there is a smooth Lyapunov function $V$, asserting that $0\in \mathbb{R}^n$ is GAS under some ODE $\dot{x}=X(x)$. 
To refine our understanding beyond Theorem~\ref{prop:GAS:homotopy:semiflow} and towards a similar result for vector fields, we note that our desire relates to \textit{optimal transport} (OT). Let $\mu_0, \mu_1\in \mathscr{P}(\mathbb{R}^n)$ be Borel probability measures on $\mathbb{R}^n$. Now suppose that $\mathrm{d}\mu_0=\kappa_0 e^{-V}\mathrm{d}\lambda^n$ and $\mathrm{d}\mu_1=\kappa_1 e^{-V_q}\mathrm{d}\lambda^n$, for $\kappa_0, \kappa_1$ normalization constants. 
Can we transport $\mu_0$ to $\mu_1$ along sufficiently regular measures, preserving unimodality? Suppose the answer is yes, then we can construct a path of densities $[0,1]\ni s \mapsto f(x;s)$, with $f(x;0)=e^{-V(x)}$ and $f(x;1)=e^{-V_q(x)}$. Now, if $f$ is sufficiently regular, we can conclude that $\nabla \log f(x;s)$ renders $0$ GAS, for each fixed $s$. 
It is particularly interesting that the standard Gaussian measure results in the canonical ODE $\dot{x}=-x$. 
In what follows we touch upon this viewpoint. 

The quadratic ``\textit{Monge formulation}'' of OT on $\mathbb{R}^n$ is as follows, let $\mu,\nu\in \mathscr{P}_2(\mathbb{R}^n):=\{\mu\in \mathscr{P}(\mathbb{R}^n):\int \|x\|_2^2 \mathrm{d}\mu(x)<+\infty\}$, then we want to solve
\begin{equation}
\label{equ:Monge}
    \inf_{T:T\#\mu=\nu}\int_{\mathbb{R}^n} \|x-T(x)\|_2^2 \mathrm{d}\mu(x), 
\end{equation}
where $T\# \mu$ denotes the \textit{pushforward} of $\mu$. Brenier \cite{ref:brenier1991polar} showed that if $\mu \ll \lambda^n$, then, there is a convex map $\phi:\mathbb{R}^n\to \mathbb{R}$, being $\mu$-a.e. differentiable, such that $\nabla \phi\# \mu = \nu$ solves~\eqref{equ:Monge}. We do a simple example.

\begin{example}[Gaussian optimal transport]
\label{ex:Gaussian:OT}
\begin{upshape}
    Given two zero-mean Gaussian measures $\mu_0$ and $\mu_1$ on $\mathbb{R}^n$, with covariance matrices $\Sigma_0$ and $\Sigma_1$, the optimal map in the sense of~\eqref{equ:Monge} is simply $x\mapsto T^{\star}(x):=Ax$, for $A:=\Sigma_0^{-1/2}(\Sigma_0^{1/2}\Sigma_1\Sigma_0^{1/2})^{1/2}\Sigma_0^{-1/2}$, \textit{e.g.}, see~\cite[Rem. 2.31]{ref:peyre2019computational}. Then, for $\mu(\cdot;s):=T(\cdot;s)\# \mu_0$ with $T(\cdot;s):=(1-s)\mathrm{id}_{\mathbb{R}^n}+sT$ and $s\in [0,1]$ we have that 
    \begin{equation*}
        s\mapsto\Sigma(s) := \Sigma_0^{-1/2}\left((1-s)\Sigma_0 + s(\Sigma_0^{1/2}\Sigma_1\Sigma_0^{1/2})^{1/2} \right)^2\Sigma_0^{-1/2}\succ 0.
    \end{equation*}
    Now consider the density $\mathbb{R}^n\ni x\mapsto\rho(x;s) := \kappa(s)^{-1}e^{-\frac{1}{2}\langle x, \Sigma(s)^{-1}x\rangle}$ for $\kappa(s)^2:= (2\pi)^n\mathrm{det}(\Sigma(s))$ and set $V(\cdot;s):=-\log \rho(\cdot;s)$. Let $\Sigma_1=I_n$, then it follows that the family of ODEs $\dot{x}=-\nabla V(x;s)$, comprises a homotopy from $\dot{x}=-\nabla V(x;0)$ to $\dot{x}=-x$, through vector fields such that $0$ is GAS (\textit{e.g.}, consider the Lyapunov function $x\mapsto \widetilde{V}(x;s):=\frac{1}{2}\langle x, \Sigma(s)^{-1}x\rangle$). 
    \exampleEnd
    \end{upshape}
\end{example}

\section{Conclusion and future work}
We have provided a step towards better understanding Conley's converse question in some generality, yet, many open problems remain. Although directly working with flows has benefits, \textit{e.g.}, see~\cite{ref:aguiar2023universal}, the main open problem is the extension to vector fields and generic attractors. Several other questions are as follows. 

\noindent{{Open problem 1}}: \textit{characterize stability of~\eqref{equ:flow:homotopy:sontag:coron} throughout the homotopy.}

\noindent{{Open problem 2}}: \textit{prove or disprove that Proposition~\ref{prop:GAS:homotopy:semiflow} holds for $n=5$. A counterexample would disprove the smooth Poincar\'e conjecture for $\mathbb{S}^4$.}

Although we work with semiflows, {to leverage converse Lyapunov theory we rely on vector fields generating them}. Removing this condition is non-trivial, illustrated by~\cite[Sec. 7]{ref:fathi2019smoothing}, but also not futile, as illustrated below.  

\begin{example}[An attractor on the mapping torus]
\begin{upshape}
    To construct an example of a semiflow that does not correspond to a vector field, one can appeal to maps with a ``negative orientation'', \textit{e.g.}, a smooth vector field $X$ always results in a flow $\varphi(\cdot;X)$ such that the diffeomorphism $\varphi^t(\cdot;X)$ is isotopic to the identity\footnote{Consider the homotopy $[0,1]\ni s\mapsto \varphi^{\tau(1-s)}$ from $\varphi^{\tau}$ to $\varphi^0=\mathrm{id}$.}. To that end, consider the map $\mathbb{R}\ni x \mapsto f(x):=-\alpha x$, for some $\alpha\in (0,1)$, and the define the mapping torus of $(\mathbb{R},f)$ through $M_f := \{(x,t)\in \mathbb{R}\times [0,1]\}/{\sim}$ for $(x,1)\sim (f(x),0)$. Now, the \textit{suspension} of $f$ (under the ceiling function $c(x)\equiv 1$) is the semiflow $\varphi(\cdot;f)$ defined through $\varphi^t((x,s);f):=(f^n(x),s')$ where the pair $(n,s')$ satisfies $n+s'=t+s$ with $0\leq s'\leq 1$~\cite[Sec. 1.11]{ref:brin2002introduction}. It follows that $\{0\}\times [0,1]/{\sim}$ is an attractor. \exampleEnd
    \end{upshape}
\end{example}

\noindent{{Open problem 3}}: \textit{prove or disprove that Proposition~\ref{prop:GAS:homotopy:semiflow} holds for all semiflows that render $0$ GAS ({not just the ones generated by vector fields).}}

Section~\ref{sec:OT} touched upon connections with OT. Motivated by work in the context of geometry processing~\cite{ref:solomon2019optimalV}, we believe that more work is warranted.  

\noindent{{Open problem 4}}: \textit{elucidate what OT can tell us about the existence of stability preserving homotopies on the level of vector fields, and vice versa.}

Our work also benefits from more explicit results, \textit{e.g.}, see~\cite{ref:bramburger2021deep}.  

\appendix

\section{Continuation and the Conley index}
\label{sec:Conley}
We aim to contribute to better understanding to what extent Hopf's degree theorem~\cite[p. 51]{ref:milnor65} extends to equivalence classes of dynamical systems. The most fruitful framework in this regard is Conley's theory~\cite{ref:conley1978isolated} and we briefly highlight his index theory to elucidate our central question. 

In the context of semiflows, Conley's index theory can be set up as follows~\cite[Ch. 1]{ref:rybakowski1987homotopy}. Given a semiflow $\varphi$ on $M^n$, then, $S\subseteq M^n$ is said to be an \textit{isolated (positively) invariant set} when there is a compact set $K\subseteq M^n$, called an \textit{isolating neighbourhood}, such that $S = \mathrm{I}(K,\varphi):=\{p\in K:\varphi^t(p)\in K\,\forall t\geq 0\}\subseteq \mathrm{int}(K)$. Now, a pair of compact sets $(N,L)\subset M^n\times M^n$ is said to be an \textit{index pair} for $S$ when (i) $S=\mathrm{I}(\mathrm{cl}(N\setminus L),\varphi)$ and $N\setminus L$ is a neighbourhood of $S$; (ii) $L$ is positively invariant in $N$; and (iii) $L$ is an exit set for $N$.
Then, the \textit{Conley index} of $S$ is the homotopy type of the pointed space $(N/L,[L])$. This index is independent of the choice of index pair. Now, if some $N$ can be chosen to be an isolating neighbourhood \textit{throughout} a homotopy $[0,1]\ni s \mapsto \varphi(\cdot;s)$, then, the Conley index is preserved along that homotopy and we speak of \textit{{continuation}} (of the index). The existence of the homotopy implies continuation, but to what extent does an equivalent Conley index imply the existence of an index preserving homotopy? This is the starting point of this note and this is why we speak of a ``\textit{converse question}''. 

{
\section{Added in proof}
Recently, Kvalheim and Sontag provided a generalized Hartman-Grobman theorem that allows for resolving open problem 2 \cite[Thm. 2]{ref:kvalheim2025global} (\textit{e.g.}, first homotope to a smooth flow), see also \cite{ref:jongeneel2025HG}. Moreover, Kvalheim addressed the extension to vector fields with outstanding generality \cite{ref:kvalheim2025differential}, see \cite[\S{8.4}]{ref:kvalheim2025differential} for an example that affirms our hopes at the end of Section \ref{sec:further}.
}

\pagestyle{basicstyle}
\addcontentsline{toc}{section}{Bibliography}
\subsection*{Bibliography}
\printbibliography[heading=none]


\end{document}